\documentclass[12pt,reqno]{amsart}
\usepackage[left=1in,right=1in,top=1in,bottom=1in]{geometry}
\usepackage{amssymb}
\usepackage{amsmath}
\usepackage{mathrsfs}
\usepackage{amsfonts}
\usepackage{hyperref}
\usepackage{xcolor}
\usepackage{bbm}
\usepackage{tikz}
\usepackage{array}
\usepackage{multirow}
\usepackage{tikz}
\usepackage{ytableau}
\usepackage{thmtools}
\usepackage{thm-restate}
\usepackage{varwidth}
\usepackage{comment}
\usepackage{mathtools}
\usepackage[mathscr]{euscript}
\usepackage{bbm}
\usepackage{multicol}
\usepackage{enumerate}
\usepackage{graphicx} 
\theoremstyle{definition}
\newtheorem{thm}{Theorem}[section]
\newtheorem{lem}[thm]{Lemma}
\newtheorem*{lem*}{Lemma}
\newtheorem*{thm*}{Theorem}
\newtheorem*{cor*}{Corollary}
\newtheorem{prop}[thm]{Proposition}

\newtheorem{cor}[thm]{Corollary}

\newtheorem{defn}[thm]{Definition}
\newtheorem*{remark*}{Remark}
\newtheorem{remark}{Remark}
\newtheorem{example}{Example}

\newtheorem{cor/defn}[thm]
{Corollary/Definition}
\DeclareMathOperator{\bF}{\mathbb{F}}
\DeclareMathOperator{\cH}{\mathcal{H}}
\DeclareMathOperator{\cB}{\mathcal{B}}

\DeclareMathOperator{\SYT}{\mathrm{SYT}}
\DeclareMathOperator{\Par}{\mathbb{Y}}

\usepackage[
    backend=bibtex,
    maxnames=5,
    style=alphabetic
  ]{biblatex}
\title{Non-archimedean Infinite Hecke Algebra}

\author{Milo Bechtloff Weising}
\address{Department of Mathematics (0123),
460 McBryde Hall, Virginia Tech,
225 Stanger Street,
Blacksburg, VA 24061-1026}
\email{milojbw@vt.edu}

\date{\today}

\begin{document}

\maketitle 

\begin{abstract}
    We study the representation theory of the infinite type A Hecke algebra over a non-archimedean field in the case where the parameter is a pseudo-uniformizer. Specifically, we consider a family of representations, called almost-symmetric, which satisfy additional topological and algebraic constraints. We construct a family of irreducible almost-symmetric representations indexed by integer partitions which arise as topological completions of specific direct limits of Hecke-Specht modules. Our main result is that every irreducible almost-symmetric representation contains one of these constructed irreducibles as a dense submodule. We give detailed analysis of these representations and construct functionals analogous to finite Hecke algebra traces. 
\end{abstract}

\tableofcontents

\section{Introduction}

The study of symmetric groups is a meeting point for representation theory and combinatorics. This, too, is true for the infinite symmetric group $\mathfrak{S}_{\infty}$ but most approaches to understanding this group involve analytic techniques \cite{BO16}. As the Edrei-Thoma Theorem shows, this is to some extent unavoidable since the harmonic analysis of $\mathfrak{S}_{\infty}$ is intimately tied to deep subjects in analysis. Even though combinatorial methods are crucial for the character theory of the infinite symmetric group, the well-behaved explicit representations of $\mathfrak{S}_{\infty}$ are often not combinatorial in nature \cite{Okounkov97} \cite{Olshanski03}. In fact, some of the most natural combinatorial representations of $\mathfrak{S}_{\infty}$ are somewhat counter-intuitively poorly behaved.

The structure of type $A$ finite Hecke algebras over $\mathbb{C}$ depends heavily on the value of their parameter $t \in \mathbb{C}.$ Generically, they admit a nearly identical representation theory as the symmetric groups. Namely, for generic $t \in \mathbb{C}$, they are semisimple and their irreducible representations are indexed by integer partitions with dimensions given by counting standard Young tableaux. As a result, the harmonic analysis of the infinite type $A$ Hecke algebras over $\mathbb{C}$ are in some ways analogous to their symmetric group counterpart for $t \in \mathbb{C}$ generic \cite{VK98} \cite{GKV14}.

Recently, coming from the study of double affine Hecke algebras and the combinatorics of LLT and Macdonald polynomials, it  has become useful to consider Hecke algebras over non-archimedean fields \cite{Ion_2022}\cite{IonWu25}\cite{BW_Delta}\cite{BW25}\cite{BHMPS}. Typically, the field in question is $\mathbb{Q}((t))$ where $t$ is considered both as the topological generator of the Laurent series field $\mathbb{Q}((t))$ and as the parameter of the Hecke algebras. For finite Hecke algebras the presence of the non-archimedean topology does not make any significant difference to the underlying representation theory. However, there seem to be more important differences in the case of the infinite Hecke algebra. Specifically, the non-archimedean topology allows for the construction of new projection operators using partially-trivial idempotent elements. These results suggest that the structure of the infinite Hecke algebra might be significantly different over archimedean and non-archimedean fields at least when the parameter $t$ is a pseudo-uniformizer. 

The goal of this paper is to initiate the study of the infinite Hecke algebra $\cH_{\infty}$ over non-archimedean fields $\bF$. Because the entire representation theory of $\cH_{\infty}$ seems to be intricate, we opt to instead study a class of representations which we call almost-symmetric. These are continuous actions of $\cH_{\infty}$ on non-archimedean Banach spaces which admit an extended nontrivial continuous action by idempotent elements $\epsilon_{k}$. Note that over non-archimedean fields infinite dimensional Banach spaces are not reflexive and so we are forced to consider actions on Banach spaces as opposed to the usual situation over $\mathbb{C}$ where one considers actions on Hilbert spaces. For every partition $\lambda$, we build an almost-symmetric representation $\hat{S}_{\lambda}$ as a norm completion of the direct limit of Hecke-Specht modules $S_{(n-|\lambda|,\lambda)}$ as $n \rightarrow \infty.$ Using a combination of combinatorial and analytic methods, in Proposition \ref{prop inf specht are irred} we show that the representations $\hat{S}_{\lambda}$ are irreducible. Our main result, Theorem \ref{thm classification of irreps}, is as follows:

\begin{thm*}
    If $V$ is an irreducible almost-symmetric $\widehat{\cH}_{\infty}$-module, then for some $\lambda \in \Par$ there exists a bounded $\widehat{\cH}_{\infty}$-module embedding $ \hat{S}_{\lambda}\rightarrow  V$ with dense image.
\end{thm*}

The combinatorics involved in these computations are mostly standard but they are highly related to the notions of representation stability of Church--Ellenberg--Farb \cite{CEF15} and results of Murnaghan \cite{M_1938}. Analytically, however, due to a lack of some of the standard tools of functional analysis in the non-archimedean setting we need to use some arguments which may be non-standard to a combinatorialist and more combinatorial than an analyst might expect.

In the last section of the paper, we consider the problem of assigning trace functionals to the irreducible representations in the way one would do for finite Hecke algebra representations. Hecke algebra traces are related to chromatic quasi-symmetric functions, immanants, LLT polynomials, and Kazhdan--Lusztig theory \cite{Haiman} \cite{SW16} \cite{MSW}. Theorem \ref{thm unique inner product} shows that there is an essentially unique inner product (that is to say a non-degenerate continuous symmetric bilinear form) $(-,-)_{\lambda}$ on each of the spaces $\hat{S}_{\lambda}$ such that the generators $T_i$ of $\cH_{\infty}$ are self-adjoint. We then use the bilinear form $(-,-)_{\lambda}$ to define a linear functional $\Gamma_{\lambda}$ on $\cH_{\infty}$ which we call the regularized trace of $\hat{S}_{\lambda}.$ There are many such possible definitions for this functional so we make our definition to carefully align with the author's previous work \cite{BWMurnaghan} where highly related infinite series appeared in the context of stable-limits of the vector-valued DAHA representations of Dunkl--Luque \cite{DL_2011}. We prove in Theorem \ref{thm integrality} that the values of $\Gamma_{\lambda}$ at integral elements of $\cH_{\infty}$ are still integral. Lastly, we prove that the values $\Gamma_{\lambda}(1)$ are all rational expressions of $t$ (Theorem \ref{rationality thm}). Putting these results together imply that the values $\Gamma_{\lambda}(T_{\sigma})$ are always rational expressions in $t$ with norm $\leq 1.$ An arbitrary renormalization would not be expected to satisfy these special properties at least partially justifying our choice. These results suggest that there should be a well-defined Hecke algebra trace theory for the infinite Hecke algebra over a non-archimedean field given the appropriate choice of normalization.

\section{Background}

\subsection{Combinatorics}

Here we define the main combinatorial notions we require for this paper.

\begin{defn}
    For $n \geq 1$ let $\mathfrak{S}_n$ denote the \textbf{symmetric group} of permutations of the set $\{1,\dots,n\}.$ Let $s_1,\dots,s_{n-1}$ denote the simple transpositions $s_{i}:=(i~~i+1).$ Given $\sigma \in \mathfrak{S}_n$ define the length of $\sigma$, $\ell(\sigma)$, as the minimal number of $s_i$'s required to express $\sigma$ as a product of the form $\sigma = s_{i_1}\dots s_{i_{\ell}}.$ Equivalently, $\ell(\sigma)$ is the number of inversions of $\sigma$, i.e., pairs of numbers $1\leq i< j\leq n $ such that $\sigma(i) > \sigma(j).$ We say that $\sigma = s_{i_1}\dots s_{i_{\ell}}$ is a reduced expression if $\ell = \ell(\sigma).$ For all $n \geq 1$ we embed $\mathfrak{S}_n \subset \mathfrak{S}_{n+1}$ as the subgroup of all $\sigma \in \mathfrak{S}_{n+1}$ such that $\sigma(n+1) = n+1.$ Define the \textbf{infinite symmetric group} $\mathfrak{S}_{\infty}$ as the direct limit $\lim_{\rightarrow} \mathfrak{S}_n.$ Equivalently, $\mathfrak{S}_{\infty}$ is the group of all bijections $\sigma: \{1,2,\dots\} \rightarrow \{1,2,\ldots\}$ such that $\sigma(i) = i$ for all but finitely many $i\geq 1.$ In this paper, a composition is a tuple $\mu = (\mu_1,\dots,\mu_r)$ where $\mu_i \geq 1.$ The size of the composition is $|\mu|:= \mu_1+\dots+\mu_r$ and the length of the composition is $\ell(\mu):= r.$ An infinite composition refers to a tuple of the form $\gamma = (\mu_1,\cdots,\mu_r,\infty)$ where $\mu$ is a (finite) composition and $\infty$ is a formal symbol. Given a (finite) composition $\mu = (\mu_1,\dots,\mu_r)$ with size $|\mu| = n$ define the \textbf{Young subgroup} $\mathfrak{S}_{\mu}$ as the copy of $\mathfrak{S}_{\mu_1}\times \dots \times \mathfrak{S}_{\mu_r}$ embedded into $\mathfrak{S}_n$ such that each $\mathfrak{S}_{\mu_i}$ permutes the labels $\{\mu_1+\dots\mu_{i-1}+1,\dots,\mu_1+\dots\mu_{i-1}+\mu_i\}.$ Given an infinite composition $\gamma = (\mu_1,\dots,\mu_r,\infty)$ we consider the Young subgroup $\mathfrak{S}_{\gamma}$ as the copy of $\mathfrak{S}_{\mu}\times \mathfrak{S}_{\infty}$ where $\mathfrak{S}_{\mu}$ permutes the set $\{1,\dots,n\}$ and $\mathfrak{S}_{\infty}$ permutes the set $\{n+1,n+2,\dots\}.$
\end{defn}

\begin{defn}\label{defn partitions}
    A \textbf{partition} is a weakly decreasing sequence of positive integers $\lambda = (\lambda_1,\dots,\lambda_{\ell})$. Define $\Par$ as the set of all partitions. We refer to $\ell(\lambda):= \ell$ as the length of $\lambda$ and $|\lambda|:= \lambda_1+\dots + \lambda_{\ell}$ as the size of $\lambda.$ We identify every partition $\lambda$ with its corresponding \textbf{Young diagram} in French notation, i.e., the left-adjusted lattice diagram whose row lengths from top to bottom are $\lambda_1,\dots,\lambda_{\ell}.$ The \textbf{dominance ordering} on $\Par$ is given by $\mu \trianglelefteq \lambda$ if for all $\mu_1+\dots+\mu_i \leq \lambda_1+\dots + \lambda_i$ where we consider $\mu_i = 0$ if $i > \ell(\mu)$ and $\lambda_i = 0$ if $i > \ell(\lambda).$
\end{defn}

\begin{defn}\label{defn Young diagrams}
     A \textbf{standard Young tableau} of shape $\lambda \in \Par$ is a labeling $\tau: \lambda \rightarrow \{1,2,\dots\}$ of the boxes of $\lambda$ which is strictly increasing from left to right and strictly increasing from top to bottom and takes the values $\{1,\dots,|\lambda|\}.$ We write $\SYT(\lambda)$ for the set of all standard Young tableaux of shape $\lambda.$ Let $\tau_{\lambda}$ denote the \textbf{column-standard} Young tableau of shape $\lambda$ which is the unique element of $\SYT(\lambda)$ such that if $\square_1$ is directly below $\square_2$, then $\tau_{\lambda}(\square_1) = \tau_{\lambda}(\square_2) -1.$ If $\tau \in \SYT(\lambda)$ and $1\leq i \leq n-1$ such that swapping the labels $i,i+1$ results in a standard Young tableau, we write $s_i(\tau)\in \SYT(\lambda)$ for this resulting tableau. We give $\SYT(\lambda)$ the structure of a partially ordered set uniquely via the following rules:
    \begin{itemize}
        \item If we have some $\tau \in \SYT(\lambda)$ and $1\leq i \leq n-1$ such that the entry $i+1$ appears strictly to the right and above the entry $i$ in $\tau$, then we say $s_i(\tau) > \tau$ is a cover for the ordering $>.$
        \item The order $>$ on $\SYT(\lambda)$ is the transitive closure of the above cover relations. 
    \end{itemize}
\end{defn}

\begin{example}
    The following is a maximal length chain in $\SYT(2,2,1):$

    $$\begin{ytableau}
        1 & 4 \\
        2 &  5 \\
        3 & \none \\
        \end{ytableau} < \begin{ytableau}
        1 & 3 \\
        2 &  5 \\
        4 & \none \\
        \end{ytableau}< \begin{ytableau}
        1 & 2 \\
        3 &  5 \\
        4 & \none \\
        \end{ytableau} < \begin{ytableau}
        1 & 2 \\
        3 &  4 \\
        5 & \none \\
        \end{ytableau}.$$
    Note the column-standard tableau is the unique minimal element and the row-standard tableau is the unique maximal element of $\SYT(2,2,1)$.
\end{example}

\begin{defn}
    Let $\lambda \in \Par$ with size $|\lambda| = n$ and $\tau \in \SYT(\lambda).$ For all $\square \in \lambda$ define the \textbf{content} $c(\square):= j-i$ where $\square = (i,j).$ For $1\leq i \leq n$ define $c_{\tau}(i):= c(\square)$ where $\square \in \lambda$ is the unique square such that $\tau(\square) = i.$
\end{defn}

\begin{example}
    For $\tau = \begin{ytableau}
        1 & 3 \\
        2 &  5 \\
        4 & \none \\
        \end{ytableau} \in \SYT(2,2,1),$ the content values are $$(c_{\tau}(1),c_{\tau}(2),c_{\tau}(3),c_{\tau}(4),c_{\tau}(5)) = (0,-1,1,-2,0).$$  
\end{example}

\subsection{Finite Hecke Algebras}

For the remainder of the paper, fix a \textbf{non-archimedean field} $\mathbb{F}$ which is complete with respect to the non-trivial norm $|\cdot|$ and fix $t \in \mathbb{F}$ to be a \textbf{pseudo-uniformizer}, i.e., $0< |t| < 1.$ 

\begin{defn}\label{defn finite Hecke}
     Define $\mathcal{H}_n$ to be the \textbf{finite Hecke algebra} over $\mathbb{F}$ with generators $T_1,\ldots,T_{n-1}$ satisfying 
    \begin{itemize}
        \item $(T_i-1)(T_i+t) = 0$
        \item $T_iT_{i+1}T_i = T_{i+1}T_iT_{i+1}$
        \item $T_iT_j = T_jT_i$ for $|i - j| > 1.$
    \end{itemize}
    Given a composition $\mu = (\mu_1,\dots,\mu_r)$ of size $|\mu| = n$ define the \textbf{Young subalgebra} $\cH_{\mu}$ as the copy of $\cH_{\mu_1}\otimes \cdots \otimes \cH_{\mu_r}$ embedded in $\cH_{n}$ where each $\cH_{\mu_i}$ is generated by the elements $\{T_{\mu_1+\dots+\mu_{i-1}+1},\dots,T_{\mu_1+\dots+\mu_{i-1}+\mu_i -1}\}.$ Define the trivial representation $\mathrm{triv}(\mu)$ as the $1$-dimensional representation of $\cH_{\mu}$ determined by the algebra character $T_i \mapsto 1.$
\end{defn}

Over the complex numbers $\mathbb{C}$ the algebra $\mathcal{H}_n$ as defined above is \textbf{semisimple} so far as $t \neq 0$ is not a root of unity of sufficiently small order. In the non-archimedean setting, we refer to a more general semi-simplicity result of Andersen--Stroppel--Tubbenhauer. We are only using a special case of their result: particularly, no matter the characteristic of $\mathbb{F}$, the element $t$ is a unit but cannot be a root of unity.

\begin{thm}[Andersen--Stroppel--Tubbenhauer] \cite{AST17}\label{AST thm}
    The algebra $\mathcal{H}_n$ is semisimple over $\mathbb{F}.$
\end{thm}

Here we give an explicit construction of the \textbf{irreducible} representations of $\cH_n.$

\begin{defn}\label{defn JM elements}
We define the \textbf{Jucys-Murphy} elements $\theta_1,\ldots, \theta_n \in \cH_n$ by $\theta_1 := 1$ and $\theta_{i+1}:= tT_{i}^{-1}\theta_{i}T_{i}^{-1}$ for $1\leq i \leq n-1$. Further, define the \textbf{intertwiner} elements $\varphi_1,\ldots, \varphi_{n-1}$ by $\varphi_i: = (tT_{i}^{-1})\theta_i - \theta_{i}(tT_i^{-1}).$ For a permutation $\sigma \in \mathfrak{S}_n$ and a reduced expression $\sigma = s_{i_1}\cdots s_{i_r}$ we write $T_{\sigma} := T_{i_1}\cdots T_{i_r}.$
\end{defn}

\begin{remark}
    There are natural algebra inclusions $\cH_{n} \rightarrow \cH_{n+1}$ given by $T_i \rightarrow T_i$ for $1\leq i \leq n-1$. Under this embedding $\theta_i \rightarrow \theta_i$ for $1 \leq i \leq n$ so we can naturally identify the copies of $\theta_i$ in both $\cH_n$ and $\cH_{n+1}$.
\end{remark}

The following list of relations is straightforward to verify.

\begin{prop}\label{additional relations for finite hecke}
    The following relations hold:
    \begin{itemize}
        \item $\theta_i = t^{i-1}T_{i-1}^{-1}\cdots T_{1}^{-1}T_{1}^{-1}\cdots T_{i-1}^{-1}$ for $1\leq i \leq n$
        \item $\theta_i\theta_j = \theta_j\theta_i$ for $1\leq i,j \leq n$
        \item $T_i\theta_j = \theta_jT_i$ for $j \notin \{i,i+1\}$
        \item $\varphi_i = tT_{i}^{-1}(\theta_i - \theta_{i+1}) + (t-1)\theta_{i+1}$ for $1 \leq i \leq n-1$
        \item $\varphi_i\varphi_{i+1}\varphi_i=\varphi_{i+1}\varphi_i\varphi_{i+1}$ for $1\leq i \leq n-1$
        \item $\varphi_i\varphi_j = \varphi_j\varphi_i$ for $|i-j| >1$
        \item $\varphi_{i}\theta_j = \theta_{s_i(j)}\varphi_{i}$ for $1\leq i \leq n-1$ and $1 \leq j \leq n$
        \item $\varphi_i^{2} = (t\theta_i-\theta_{i+1})(t\theta_{i+1}-\theta_{i}).$
    \end{itemize}
\end{prop}

The following definition gives a description of the irreducible representations of $\cH_n$. There are many equivalent methods for defining these representations but we choose to specify eigenvectors for the Jucys-Murphy elements $\theta_i$ directly as we require these eigenvectors throughout this paper.

\begin{defn}\label{irreps for finite hecke defn}
    Let $\lambda \in \Par$ with $|\lambda| = n$. Define the \textbf{Hecke-Specht module} $S_{\lambda}$ as the $\cH_n$-module spanned by $e_{\tau}$ for $\tau \in \SYT(\lambda)$ defined by the following relations:
    \begin{itemize}
        \item $\theta_i(e_{\tau}) = t^{c_{\tau}(i)}e_{\tau}$
        \item If $s_i(\tau) > \tau,$ then $\varphi_i(e_{\tau}) = (t^{c_{\tau}(i)}-t^{c_{\tau}(i+1)})e_{s_i(\tau)}.$
        \item If the labels $i,i+1$ are in the same row in $\tau$, then $T_i(e_{\tau}) = e_{\tau}.$
        \item If the labels $i,i+1$ are in the same column in $\tau$, then $T_i(e_{\tau}) = -te_{\tau}.$
    \end{itemize}
\end{defn}

\begin{remark}\label{remark explicit action}
    Using the relations from Proposition \ref{additional relations for finite hecke}, we find the following more explicit form for the action of the $T_i$ on the $\SYT(\lambda)$ basis:
\begin{itemize}
    \item If $s_i(\tau)> \tau,$ then 
    $$T_i(e_{\tau}) = e_{s_{i}(\tau)} + \frac{1-t}{1-t^{c_{\tau}(i+1)-c_{\tau}(i)}} e_{\tau}.$$
    \item If $s_i(\tau) < \tau,$ then 
    $$T_i(e_{\tau}) = t\frac{(1-t^{c_{\tau}(i) - c_{\tau}(i+1) -1})(1-t^{c_{\tau}(i)-c_{\tau}(i+1)+1})}{(1-t^{c_{\tau}(i)-c_{\tau}(i+1)})^2} e_{s_{i}(\tau)} - t^{c_{\tau}(i)-c_{\tau}(i+1)} \frac{1-t}{1-t^{c_{\tau}(i)-c_{\tau}(i+1)}}e_{\tau}.$$
    \item If the labels $i,i+1$ are in the same row in $\tau$, then $T_i(e_{\tau}) = e_{\tau}.$
        \item If the labels $i,i+1$ are in the same column in $\tau$, then $T_i(e_{\tau}) = -te_{\tau}.$
\end{itemize}
\end{remark}

\begin{example}
For the following examples, for the sake of notation, we conflate $\tau$ with $e_{\tau}:$
\begin{itemize}
    \item $$T_2~\ytableausetup
 {mathmode, boxframe=normal, boxsize=1em}\begin{ytableau}
    1 & 2 \\
    3 & \none \\
\end{ytableau} = \begin{ytableau}
    1 & 3 \\
    2 & \none \\
\end{ytableau}
+ \left(\frac{1-t}{1-t^2} \right) \begin{ytableau}
    1 & 2 \\
    3 & \none \\
\end{ytableau}$$

\item  $$T_4~\ytableausetup
 {mathmode, boxframe=normal, boxsize=1em}\begin{ytableau}
    1 & 2 & 5 \\
    3 & 6 & \none \\
    4 & \none & \none \\
\end{ytableau}
 = \begin{ytableau}
    1 & 2 & 4 \\
    3 & 6 & \none \\
    5 & \none & \none \\
\end{ytableau}
+ \left(\frac{1-t}{1-t^4} \right) \begin{ytableau}
    1 & 2 & 5 \\
    3 & 6 & \none \\
    4 & \none & \none \\
\end{ytableau}$$

\item $$T_1~\ytableausetup
 {mathmode, boxframe=normal, boxsize=1em}\begin{ytableau}
    1 & 2  \\
    3 & 4  \\
\end{ytableau}
 = \begin{ytableau}
    1 & 2  \\
    3 & 4  \\
\end{ytableau}$$

\item $$T_3~\ytableausetup
 {mathmode, boxframe=normal, boxsize=1em}\begin{ytableau}
    1 & 2  \\
    3 & 5  \\
    4 & \none  \\
\end{ytableau}
 = -t~\begin{ytableau}
    1 & 2  \\
    3 & 5  \\
    4 & \none  \\
\end{ytableau}$$

\end{itemize}

\end{example}

\begin{prop}\label{prop finite hecke irreps}
    Definition \ref{irreps for finite hecke defn} is well-posed, i.e., the action of the operators $T_i$ on $S_{\lambda}$ define an irreducible $\cH_n$-module.
\end{prop}
\begin{proof}
    This is a standard result which we leave to the reader.
\end{proof}

\begin{remark}\label{remark branching}
  The set $\{ \lambda \in \Par: |\lambda| = n\}$ gives a full set of irreducible $\cH_n$-modules up to isomorphism. Note that for $\tau,\tau' \in \SYT(\lambda)$, the $\theta$-weights of $e_{\tau}$ and $ e_{\tau'}$ are equal if and only if $\tau = \tau'.$ The algebras $\cH_{n-1} \subset \cH_{n}$ satisfy the typical branching rules for symmetric group Specht-modules, namely:
  $$\mathrm{Res}_{\cH_{n-1}}^{\cH_{n}} S_{\lambda} = \bigoplus_{\mu\subset \lambda} S_{\mu}$$ where $|\lambda| = n$ and $\mu\subset \lambda$ range over all of those partitions with size $|\mu| = n-1$ contained in $\lambda.$
\end{remark}

We require the following special elements of the finite Hecke algebras.

\begin{defn}\label{defn partial trivial idem}
    For $1 \leq k \leq n$ define the \textbf{partially-trivial idempotent} $\epsilon_{k}^{(n)} \in \cH_n$ as 
    $$\epsilon_{k}^{(n)}:= \frac{1}{[n-k]_t!} \sum_{\sigma \in \mathfrak{S}_{1^k,n-k}} t^{\binom{n-k}{2} -\ell(\sigma)} T_{\sigma}.$$
\end{defn}

\begin{lem}\label{lem idem relations}
    For all $0 \leq k, \ell \leq n$, $\epsilon_{k}^{(n)}\epsilon_{\ell}^{(n)} = \epsilon_{\min\{k,\ell\}}^{(n)}.$ For $0 \leq k \leq n$, $1 \leq i \leq k-1$, and $k+1 \leq j \leq n-1,$ $T_i \epsilon_{k}^{(n)} = \epsilon_{k}^{(n)} T_i$ and $T_j \epsilon_{k}^{(n)} = \epsilon_{k}^{(n)} T_j = \epsilon_{k}^{(n)}.$
\end{lem}
\begin{proof}
    These relations are straightforward to check directly.
\end{proof}

\subsection{Non-archimedean Analysis}

For an overview of the basics of non-archimedean functional analysis we refer the reader to \cite{Schneider}.

\begin{defn}
    A \textbf{Banach space} over $\bF$ is a pair of a vector space $V$ over $\bF$ and a norm $||\cdot||$ such that $V$ is complete with respect to the metric induced by $||\cdot||$ and $||\cdot||$ satisfies the \textbf{ultra-metric} condition, i.e., for all $v,w \in V,$ $||v+w || \leq \max\{ ||v||,||w||\}.$
\end{defn}

We often label Banach spaces $(V,||\cdot||)$ simply by their vector space $V$ when the norm $||\cdot||$ is implicit. Note that the field $\bF$ itself is a Banach space whose norm is $|\cdot|.$ 

\begin{defn}
    Let $S$ be a countable set. Define $\ell_{\infty}(S)$ as the space of all functions $f: S \rightarrow \bF$ such that $||f||:= \sup_{j \in S} |f(j)| < \infty.$ Define $c_0(S)$ as the subspace of $f\in \ell_{\infty}(S)$ such that for all $\epsilon > 0$ there exist only finitely many $j\in S$ such that $|f(j)|\geq \epsilon.$
\end{defn}

Often, we instead write elements $f\in \ell_{\infty}(S)$ as $f \equiv \sum_{j \in S} f(j) e_{j}$ where $\{e_j\}_{j \in S}$ are indicator functions for their specified elements. For $\ell_{\infty}(S)$ when $S$ is infinite this is a formal sum but for $c_{0}(S)$ this is a genuine sum in the sense that for $f\in c_0(S),$ if $S = \{j_1,j_2,\dots\}$ is an enumeration of $S$, then $ f = \lim_{n \rightarrow \infty} \sum_{1\leq i \leq n} f(j_i) e_{j_i}.$ In other words, the indicator functions $\{e_{j}|j \in S\}$ have dense span in $c_0(S).$ This fact will be important throughout the paper.

\begin{defn}
    Let $(V_1,||\cdot||_1),(V_2,||\cdot||_2)$ be a Banach spaces over $\bF.$ A linear map $A: V_1 \rightarrow V_2$ is bounded if there exists some $C \in \mathbb{R}_{\geq 0}$ such that for all $v \in V_1,$ 
    $||A(v)||_{2} \leq C ||v||_1.$ The norm of a bounded operator $A$ is defined by 
    $||A||_{V_1,V_2}:= \sup_{v \in V_1\setminus \{0\}} \frac{||A(v)||_2}{||v||_1}$. We write $\cB(V_1,V_2)$ for the Banach space of all such bounded linear maps with norm $||\cdot||_{V_1,V_2}$. For any Banach space $(V,||\cdot||)$ we write $\cB(V):= \cB(V,V).$ 
\end{defn}

Note that all bounded linear maps between Banach spaces are topologically continuous. 

\begin{defn}
    Let $\mathcal{A}$ be an algebra over $\bF.$ A representation of $\mathcal{A}$ is an algebraic representation of $\mathcal{A}$ on a Banach space $V$ such that for all $x \in \mathcal{A}$ the action map $v \mapsto x(v)$ is bounded. We say that the representation $V$ is irreducible if the only closed $\mathcal{A}$-invariant subspaces of $V$ are $0$ and $V.$ A homomorphism $f: V_1\rightarrow V_2$ of $\mathcal{A}$-representations is a bounded linear map which is $\mathcal{A}$-equivariant.
\end{defn}

Later in the paper, we need to discuss inner products on Banach spaces over $\bF.$ Since we are working over a non-archimedean field the concept of inner products much be approached carefully.

\begin{defn}\label{defn inner products}
    Let $V$ be a Banach space over $\bF.$ An \textbf{inner product} on $V$ is a non-degenerate continuous symmetric bilinear form $\langle-,-\rangle: V \times V\rightarrow \bF$. If $A\in \cB(V)$, then we say $A$ is \textbf{adjoint-able} if there exists $A^{*} \in \cB(V)$ such that for all $v,w \in V,$ $\langle A(v),w \rangle = \langle v, A^{*}(w) \rangle.$
\end{defn}

\begin{lem}\label{lem uniqueness of adjoints}
    If $A \in \cB(V)$ is adjoint-able with respect to an inner product $\langle-,-\rangle$, then it has a unique adjoint $A^* \in \cB(V)$.
\end{lem}
\begin{proof}
    Suppose $(A^*)_1,(A^{*})_2\in \cB(V)$ both satisfy for all $v,w \in V$ that $\langle A(v),w \rangle = \langle v,(A^*)_i(w) \rangle$ for $i=1,2.$ Then for all $v,w \in V,$
    $$\langle v,\left( (A^*)_1 - (A^*)_2 \right)(w) \rangle = \langle v, (A^*)_1(w) \rangle - \langle v, (A^*)_2(w) \rangle = \langle A(v),w \rangle- \langle A(v),w \rangle =0.$$ Since $\langle - , - \rangle$ is non-degenerate $(A^*)_1(w) = (A^*)_2(w)$ for all $w \in V$ and so $(A^*)_1 = (A^*)_2.$
\end{proof}

\begin{defn}
    If $A \in \cB(V)$ is adjoint-able with respect to an inner product $\langle-,-\rangle$, then we say $A$ is \textbf{self-adjoint} if $A^* = A$ and \textbf{skew-adjoint} if $A^* = -A.$
\end{defn}

\begin{lem}\label{lem adjoint properties}
    Let $\langle-,-\rangle$ be an inner product on the Banach space $V$ over $\bF$ , let $A,B \in \cB(V)$ be adjoint-able operators, and let $\alpha \in \bF$. The following hold:
    \begin{itemize}
        \item $(A + \alpha B)^* = A^* + \alpha B^*$
        \item $(AB)^* = B^* A^*$
        \item If $A,B$ are self-adjoint, then $AB-BA$ is skew adjoint.
    \end{itemize}
\end{lem}
\begin{proof}
    This is immediate from the uniqueness of adjoint operators in Lemma \ref{lem uniqueness of adjoints} and a little algebra.
\end{proof}

\subsection{Non-archimedean Infinite Hecke Algebra}

Our primary object of study is the infinite Hecke algebra in type $A$ and a certain idempotent-completion of this algebra.

\begin{defn}\label{defn inf hecke}
    Define the \textbf{infinite Hecke algebra} $\cH_{\infty}$ as the $\bF$-algebra generated by $T_1,T_2,\dots$ subject to the relations: 
    \begin{itemize}
        \item for all $i\geq 1$, $(T_i-1)(T_i+t) = 0,$
        \item for all $i,j \geq 1$ with $j \notin \{i,i+1\},$ $T_iT_j = T_jT_i,$
        \item for all $i \geq 1,$ $T_iT_{i+1}T_i = T_{i+1}T_iT_{i+1}.$
    \end{itemize}

    For an infinite composition $(\mu,\infty)$ define $\cH_{\mu,\infty}$ as the copy of $\cH_{\mu} \otimes \cH_{\infty} \subseteq \cH_{\infty}$ where $\cH_{\mu} \subset \cH_{|\mu|}$ and $\cH_{\infty}$ is generated by $T_{|\mu|+1},T_{|\mu|+2},\dots .$ We write $\mathrm{triv}(\mu,\infty)$ for the $1$-dimensional representation of $\cH_{\mu,\infty}$ determined by the algebra character $T_i \mapsto 1$ for all relevant $i\geq 1.$

    Define the \textbf{extended infinite Hecke algebra} $\widehat{\cH}_{\infty}$ by adjoining elements $\epsilon_0,\epsilon_1,\epsilon_2,\dots$ to $\widehat{\cH}_{\infty}$ subject to the relations:
    \begin{itemize}
        \item for all $k \geq 0,$ $\epsilon_k^2 = \epsilon_k,$
        \item for all $k,\ell \geq 0,$ $\epsilon_{k}\epsilon_{\ell} = \epsilon_{\min\{k,\ell \} },$
        \item for all $k \geq 0,$ $1 \leq i \leq k-1$, and $k+1 \leq j$, $\epsilon_{k}T_i = T_i \epsilon_k$ and $\epsilon_k T_j = T_j \epsilon_{k} = \epsilon_k.$
    \end{itemize}
\end{defn}

\begin{remark}
    The addition of the idempotent elements $\epsilon_k$ to $\cH_{\infty}$ is somewhat analogous to a \textbf{von Neumann algebra} completion of $\cH_{\infty}.$ If one had access to the usual tools in the analysis of $C^{*}$-algebras over $\mathbb{C}$, one would instead find a faithful representation of $\cH_{\infty}$ on a Hilbert space $V$ and study the weak closure of $\cH_{\infty}$ in the algebra of bounded operators on $V$. This would contain orthogonal projections like the $\epsilon_k$'s. In fact, one may try to do this in our non-archimedean situation. However, there are some technical issues which get in the way. The primary issue is that, as far as this author is aware, there is no suitable analogue of the von Neumann \textbf{double commutant theorem} for non-archimedean operator algebras nor is there a standard notion of weak operator topology for non-archimedean Banach algebras. That is not to say that no such theory could exist but rather that it seems to be a gap in the literature. Instead of this purely analytic approach, we have opted in this paper to use a mix of combinatorial and analytic arguments to get over these technical constraints.
\end{remark}

In this paper, we focus on a special well-behaved class of representations for the extended infinite Hecke algebra $\widehat{\cH}_{\infty}.$

\begin{defn}\label{defn almost-symm}
    An \textbf{almost-symmetric representation} of $\widehat{\cH}_{\infty}$ is a Banach space $V$ over $\bF$ along with an algebra homomorphism $\rho: \widehat{\cH}_{\infty} \rightarrow \cB(V)$ such that 
    \begin{itemize}
        \item For all $i \geq 1,k \geq 0$, $||\rho(T_i)|| \leq 1$ and $||\rho(\epsilon_k)|| \leq 1.$
        \item There exists some $k \geq 0$ such that $\rho(\epsilon_k) \neq 0.$
    \end{itemize}
\end{defn}

We often drop the homomorphism $\rho$ notation when context is clear.

\begin{example}
    The simplest example of an almost-symmetric representation of $\widehat{\cH}_{\infty}$ is the trivial representation $\hat{S}_{\emptyset}$. This is the $1$-dimensional space $\bF.v$ with action given by $T_i(v):= v$ and $\epsilon_{k}(v):= v$ for all $i \geq 1$ and $k \geq 0.$ Clearly, $||T_i|| = ||\epsilon_k|| =1$ for all $i \geq 1$ and $k \geq 0.$ The simplest non-example is the \textbf{sign} representation which is the $1$-dimensional space $\bF.w$ with action given by $T_i(w):= -tw$ and $\epsilon_{k}(w) := 0$ for all $i \geq 1$ and $k \geq 0.$ Clearly, $||T_i|| = |t| < 1$ and $\epsilon_k =0$ for all $i \geq 1$ and $k \geq 0.$ Therefore, the sign representation is not almost-symmetric. 
\end{example}

\section{Irreducible Almost-symmetric Representations}

\subsection{Construction of Irreducible Representations}

Here we construct special representations $\hat{S}_{\lambda}$ of $\widehat{\cH}_{\infty}$ indexed by partitions $\lambda \in \Par$. We then show that representations are in fact irreducible. First, we construct certain representations $S_{\lambda^{(\infty)}}$ of $\cH_{\infty}$ whose closures with respect to certain norms yield the representations $\hat{S}_{\lambda}.$

\begin{defn} \label{defn inf specht module}
    For $\lambda \in \Par$ define $n_{\lambda}:= |\lambda| + \lambda_1$. For $n \geq n_{\lambda}$ define $\lambda^{(n)}:= (n-|\lambda|,\lambda).$ Define $\lambda^{(\infty)}$ as the infinite Young diagram $\lambda^{(\infty)}:= \cup_{n \geq n_{\lambda}} \lambda^{(n)}.$ Write write $\SYT_{\infty}(\lambda)$ for the set of bijective fillings $\tau: \lambda^{(\infty)} \rightarrow \{1,2,3,\dots\}$ which are strictly increasing down columns and left to right across rows. We give $\SYT_{\infty}(\lambda)$ the poset structure determined by the poset structures on each $\SYT(\lambda^{(n)}).$ Define the $\cH_{\infty}$-module $S_{\lambda^{(\infty)}}:=\bigcup_{n \geq n_{\lambda}} S_{\lambda^{(n)}}$ where for all $n \geq n_{\lambda}, m\geq 0$, and $\tau \in \SYT(\lambda^{(n)})$ we identify $e_{\tau} \in S_{\lambda^{(n)}}$ with the vector $e_{\tau'} \in S_{\lambda^{(n+m)}}$ for $\tau' \in \SYT(\lambda^{(n+m)})$ the unique filling such that $\tau'|_{\lambda^{(n)}} = \tau.$
\end{defn}

\begin{example}
    Consider $\tau = \begin{ytableau}
    1 & 2 & 4 \\
    3 & 6 & \none \\
    5 & \none & \none \\
\end{ytableau} \in \SYT(3,2,1).$ We can think of $(3,2,1)$ instead as $\lambda^{(6)}$ where $\lambda =(2,1).$ To obtain an element of $\SYT_{\infty}(2,1)$ we identify $\tau$ with the following elements of $\SYT(\lambda^{(7)}),\SYT(\lambda^{(8)}),\dots:$

$$\begin{ytableau}
    1 & 2 & 4 \\
    3 & 6 & \none \\
    5 & \none & \none \\
\end{ytableau} ~\equiv~ \begin{ytableau}
    1 & 2 & 4 & 7\\
    3 & 6 & \none &\none \\
    5 & \none & \none & \none \\
\end{ytableau} ~\equiv~ \begin{ytableau}
    1 & 2 & 4 & 7 & 8 \\
    3 & 6 & \none &\none &\none \\
    5 & \none & \none & \none & \none\\
\end{ytableau} ~ \equiv ~ \begin{ytableau}
    1 & 2 & 4 & 7 & 8 & 9  \\
    3 & 6 & \none &\none &\none & \none \\
    5 & \none & \none & \none & \none & \none \\
\end{ytableau}\equiv \dots .$$
\end{example}

The dimensions of the partially-trivial subspaces of the representations $S_{\lambda^{(n)}}$ satisfy useful stability properties as $n$ increases.

\begin{lem}\label{lem stable partial invar dims}
    Let $\lambda, \mu \in \Par$, $r \geq n_{\lambda}$, and $m \geq \max\{ n_{\mu}, |\mu| + \lambda_1\}.$ Then 
    $$\dim \mathrm{Hom}_{\cH_{1^r,m}}\left( \mathrm{triv}(1^r) \otimes S_{\mu^{(m)}}, S_{\lambda^{(r+m)}} \right) = \begin{cases}
       \binom{r}{|\lambda/\mu|} |\SYT(\lambda/\mu)| & \mu \subseteq \lambda\\
        0 & \text{otherwise}.
    \end{cases} $$
    In particular, for $\mu = \emptyset$ and $m \geq \lambda_1$, $\dim \mathrm{Hom}_{\cH_{1^r,m}}\left( \mathrm{triv}(1^r,m), S_{\lambda^{(r+m)}} \right) = \binom{r}{ |\lambda|} |\SYT(\lambda)|.$
\end{lem}
\begin{proof}
    By Frobenius reciprocity, 
    $$\dim \mathrm{Hom}_{\cH_{1^r,m}}\left( \mathrm{triv}(1^r) \otimes S_{\mu^{(m)}}, S_{\lambda^{(r+m)}} \right) = \dim \mathrm{Hom}_{\cH_{r+m}} \left(\mathrm{Ind}_{\cH_{1^r,m}}^{\cH_{r+m}}\mathrm{triv}(1^r) \otimes S_{\mu^{(m)}}, S_{\lambda^{(r+m)}} \right)$$ and by the branching rules in Remark \ref{remark branching}, 
    $$\dim \mathrm{Hom}_{\cH_{r+m}} \left(\mathrm{Ind}_{\cH_{1^r,m}}^{\cH_{r+m}}\mathrm{triv}(1^r) \otimes S_{\mu^{(m)}}, S_{\lambda^{(r+m)}} \right) = |\SYT( \lambda^{(r+m)}/\mu^{(m)}) |$$ where we consider the skew diagram $\lambda^{(r+m)}/\mu^{(m)}$ empty unless $\mu^{(m)} \subseteq \lambda^{(r+m)}$. Note that $\mu^{(m)} \subseteq \lambda^{(r+m)}$ is equivalent to $\mu \subseteq \lambda.$ If $\mu \not\subseteq \lambda$, then the count is $0.$ Assume $\mu \subseteq \lambda$. Since $m \geq |\mu| + \lambda_1,$ the skew diagram $\lambda^{(r+m)}/\mu^{(m)}$ is a disjoint union of a horizontal row of size $r-|\lambda/\mu|$ and the skew diagram $\lambda/\mu.$ Thus 
    $$|\SYT( \lambda^{(r+m)}/\mu^{(m)}) | = \binom{r}{r-|\lambda/\mu|} |\SYT(\lambda/\mu)| = \binom{r}{|\lambda/\mu|} |\SYT(\lambda/\mu)|.$$
\end{proof}

It will be important to concretely describe a basis for the partially-trivial subspaces of the modules $S_{\lambda^{(n)}}.$

\begin{defn}\label{defn partial inv basis}
    Let $\lambda \in \Par$ and $r \geq n_{\lambda}$. Define $C_{\lambda,r}$ denote the set of all pairs $(\mu,\beta)$ of horizontal strip tableaux $\mu$ with size $\lambda_1$ on the boundary of $\lambda^{(r+\lambda_1)}$ and $\beta$ a standard filling of the complement diagram $\lambda^{(r+\lambda_1)} \setminus \mu$. We may inject $C_{\lambda,r}$ into $\SYT(\lambda^{(r+\lambda_1)})$ by assigning a given $(\mu,\beta) \in C_{\lambda,r}$ to the unique $\gamma_{\mu,\beta} \in \SYT(\lambda^{(r+\lambda_1)})$ determined by \begin{itemize}
        \item $\gamma_{\mu,\beta} |_{\mu}$ is column standard with entries $\{r+1,\dots,r+\lambda_1\}$
        \item $\gamma_{\mu,\beta}|_{\lambda^{(r+\lambda_1)}/ \mu} = \beta.$  
    \end{itemize}
    For $(\mu,\beta) \in C_{\lambda,r}$ define 
    $$v_{\mu,\beta}:= \epsilon_{r}^{(r+\lambda_1)}(e_{\gamma_{\mu,\beta}}) \in S_{\lambda^{(r+\lambda_1)}}.$$
\end{defn}

\begin{example}\label{example complement diagrams}
    Consider $\lambda = (2,1)$ and $r = 8 \geq n_{2,1} = 5.$ Then the partitions $\nu$ such that $\lambda^{(10)}/\nu = (7,2,1)/\nu$ is a horizontal strip with size $2$ are exactly $(5,2,1),(6,2),(6,1,1)$, and $(7,1):$
    \begin{itemize}
        \item[] 
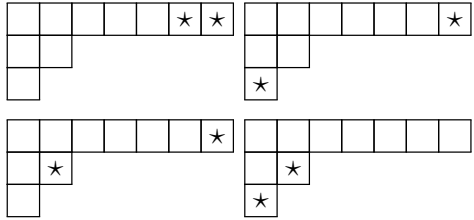
$$\begin{ytableau}
    ~&  ~& ~& ~& ~& \star & \star \\
    ~ & ~ & \none & \none& \none &  \none & \none \\
     ~& \none & \none & \none & \none & \none &\none \\
\end{ytableau} ~~~~~\begin{ytableau}
    ~&  ~& ~& ~& ~& ~& \star \\
    ~ & ~ & \none & \none& \none &  \none & \none \\
     \star & \none & \none & \none & \none & \none &\none \\
\end{ytableau} $$
        \item[] $$\begin{ytableau}
    ~&  ~& ~& ~& ~& ~& \star \\
    ~ & \star & \none & \none& \none &  \none & \none \\
     ~& \none & \none & \none & \none & \none &\none \\
\end{ytableau} ~~~~ \begin{ytableau}
    ~&  ~& ~& ~& ~& ~& ~ \\
    ~ & \star & \none & \none& \none &  \none & \none \\
     \star & \none & \none & \none & \none & \none &\none \\
\end{ytableau} $$
    \end{itemize}
Here the $\star$'s show the complementary diagram $\mu = \lambda^{(10)}/\nu.$ Choosing any standard fillings $\beta$ for the shapes $\nu$ give elements of $C_{(2,1),8}.$
\end{example}

The vectors $v_{\mu,\beta}$ described above give a basis for the partially-trivial vectors in the modules $S_{\lambda^{(n)}}.$

\begin{prop}\label{prop partial inv basis}
    The set $\{v_{\mu,\beta}| (\mu,\beta) \in C_{\lambda,r}\}$ is a basis for $\mathrm{Hom}_{\cH_{1^r,\lambda_1}}(\mathrm{triv}(1^r,\lambda_1),  S_{\lambda^{(r+\lambda_1)}})$.
\end{prop}
\begin{proof}
    From Remark \ref{remark explicit action} it is easy to see that for all $(\mu,\beta) \in C_{\lambda,r}$ the vector $e_{\gamma_{\mu,\beta}}$ generates a $\cH_{1^r,\lambda_1}$-submodule of $S_{\lambda^{(r+\lambda_1)}}$ isomorphic to $\mathrm{Ind}_{\cH_{1^r,\kappa_{\mu}}}^{\cH_{1^r,\lambda_1}} \mathrm{triv}(1^r,\kappa_{\mu}) \cong \mathrm{Ind}_{\cH_{\kappa_{\mu}}}^{\cH_{\lambda_1}} \mathrm{triv}(\kappa_{\mu})$ where $\kappa_{\mu}$ is the composition of size $\lambda_1$ formed from the sizes of the rows of $\mu$ reading bottom to top. By Frobenius reciprocity, 
    $$\dim \mathrm{Hom}_{\cH_{\lambda_1}}\left( \mathrm{Ind}_{\cH_{\kappa_{\mu}}}^{\cH_{\lambda_1}} \mathrm{triv}(\kappa_{\mu}) , \mathrm{triv}(\lambda_1) \right) = \dim \mathrm{Hom}_{\cH_{\kappa_{\mu}}}\left( \mathrm{triv}(\kappa_{\mu}) , \mathrm{triv}(\kappa_{\mu}) \right) =1$$ and furthermore, the vector $e_{\gamma_{\mu,\beta}}$, which represents the copy of $\mathrm{triv}(\kappa_{\mu})$ in $\mathrm{Ind}_{\cH_{\kappa_{\mu}}}^{\cH_{\lambda_1}} \mathrm{triv}(\kappa_{\mu})$, must map onto the $1$-dimensional generator of $\mathrm{triv}(\lambda_1)$ upon application of the trivial-symmetrizer $\epsilon_{r}^{(r+\lambda_1)}$. Therefore, for all $(\mu,\beta) \in C_{\lambda,r}$, $v_{\mu,\beta} = \epsilon_{r}^{(r+\lambda_1)}(e_{\gamma_{\mu,\beta}}) \neq 0.$ Furthermore, since $\theta_j \epsilon_{r}^{(r+\lambda_1)} = \epsilon_{r}^{(r+\lambda_1)}\theta_{j}$ for $1\leq j \leq r$, we know that each $v_{\mu,\beta}$ is still a simultaneous eigenvector for $(\theta_1,\dots,\theta_{r}).$ Since the joint $(\theta_1,\dots,\theta_{r})$-eigenvalues determine the shape of $\lambda^{(r+\lambda_1)}/\mu$ and the standard filling $\beta$, it follows that the $v_{\mu,\beta}$ are linearly independent. Finally, to see that $\{v_{\mu,\beta}| \mu \in C_{\lambda,r}\}$ is a basis for $\mathrm{Hom}_{\cH_{1^r,\lambda_1}}(\mathrm{triv}(1^r,\lambda_1),  S_{\lambda^{(r+\lambda_1)}})$ note that for any $\tau \in \SYT(\lambda^{(r+\lambda_1)})$ we may use the placement of the entries $\{r+1,\dots,r+\lambda_1\}$ to determine a boundary shape $\mu$ of $\lambda^{(r+\lambda_1)}$ and using the filling of the complementary diagram $\lambda^{(r+\lambda_1)}\setminus \mu$ to determine $\beta.$ If $\mu$ contains any two vertically connected boxes, then we know $\epsilon_{r}^{(r+\lambda_1)}(e_{\tau}) = 0.$ Otherwise, we know $(\mu,\beta) \in C_{\lambda,r}$ and so by induction using the relations from Remark \ref{remark explicit action} we are able to write $\epsilon_{r}^{(r+\lambda_1)}(e_{\tau}) = \alpha v_{\mu,\beta}$ for some $\alpha \in \bF$. Thus $\{ v_{\mu,\beta}| (\mu,\beta) \in C_{\lambda,r}\}$ spans $\epsilon_{r}^{(r+\lambda_1)}\left( S_{\lambda^{(r+\lambda_1)}} \right) \equiv \mathrm{Hom}_{\cH_{1^r,\lambda_1}}(\mathrm{triv}(1^r,\lambda_1),  S_{\lambda^{(r+\lambda_1)}})$ and is therefore a basis. 
\end{proof}

\begin{remark}
    It is an interesting combinatorial consequence of Proposition \ref{prop partial inv basis} that for all $\lambda \in \Par$ and $r \geq n_{\lambda}$, $|C_{\lambda,r}| = \binom{r}{ |\lambda|} |\SYT(\lambda)|.$ That is to say, $$\sum_{\substack{\lambda^{(r+\lambda_1)}/\nu ~\text{hor. strip}\\ |\nu| = r}} |\SYT(\nu)| = \binom{r}{|\lambda|} |\SYT(\lambda)|.$$ For a nontrivial example, take $\lambda = (2,1)$ and $r = 8 \geq n_{2,1} = 5$ as in Example \ref{example complement diagrams}. Then the partitions $\nu$ such that $\lambda^{(10)}/\nu = (7,2,1)/\nu$ is a horizontal strip with size $2$ are exactly $(5,2,1),(6,2),(6,1,1)$, and $(7,1).$ The standard tableaux counts for these partitions are $64,20,21$, and $7$ respectively. Now, $\binom{8}{3} |\SYT(2,1)| = 56\cdot 2 = 112$ and indeed we have $64+20+21+7 = 112.$ It would be interesting to find a bijective proof of this fact.
\end{remark}

Importantly, the vectors $v_{\mu,\beta}$ actually give a basis for the partially-trivial vectors in the full module $S_{\lambda^{(\infty)}}.$

\begin{cor}\label{cor stable partial inv basis}
    The set $\{v_{\mu,\beta}| (\mu,\beta) \in C_{\lambda,r}\}$ is a basis for $\mathrm{Hom}_{\cH_{1^r,\infty}}(\mathrm{triv}(1^r,\infty),  S_{\lambda^{(\infty)}}).$
\end{cor}
\begin{proof}
    We know that the embeddings $S_{\lambda^{(n)}} \subseteq S_{\lambda^{(n+1)}}$ are injective for all $n \geq n_{\lambda}$ so certainly $\mathrm{Hom}_{\cH_{1^r,m}}(\mathrm{triv}(1^r,m),  S_{\lambda^{(r+m)}}) \subseteq \mathrm{Hom}_{\cH_{1^r,m+1}}(\mathrm{triv}(1^r,m+1),  S_{\lambda^{(r+m+1)}})$ are also embeddings. By Lemma \ref{lem stability of partial triv idem}, it follows that $\mathrm{Hom}_{\cH_{1^r,\infty}}(\mathrm{triv}(1^r,\infty),  S_{\lambda^{(\infty)}}) = \mathrm{Hom}_{\cH_{1^r,\lambda_1}}(\mathrm{triv}(1^r,\lambda_1),  S_{\lambda^{(r+\lambda_1)}}) \subseteq S_{\lambda^{(r+\lambda_1)}}.$ Thus from Proposition \ref{prop partial inv basis}, we know that $\{v_{\mu,\beta}| (\mu,\beta) \in C_{\lambda,r}\}$ is a basis for the $\mathrm{triv}(1^r,\infty)$ vectors in $S_{\lambda^{(\infty)}}$.
\end{proof}

Using the basis vectors $v_{\mu,\beta}$, we are able to fully to describe the action of the partial-trivial idempotents $\epsilon_{r}^{(n)}$ on the modules $S_{\lambda^{(n)}}.$

\begin{lem}\label{lem action of partial triv idem}
    If $m \geq \lambda_1, r\geq n_{\lambda},$ and $\tau \in \SYT(\lambda^{(r+m)}),$ then there exists a unique $(\mu,\beta) \in C_{\lambda,r}$ and scalar $\alpha_r(\tau) \in \bF$ such that $\epsilon_{r}^{(r+m)}(e_{\tau}) = \alpha_r(\tau) v_{\mu,\beta}$. Furthermore, $|\alpha_r(\tau)| \leq 1.$
\end{lem}
\begin{proof}
    By considering the placement of the entries $\{r+1,\dots,r+\lambda_1\}$, any $\tau \in \SYT(\lambda^{(r+m)})$ may be obtained from a unique $\gamma_{\mu,\beta}$ with $C_{\lambda,r}$ via a sequence of steps $\tau \mapsto s_j(\tau)$ with $s_{j}(\tau) > \tau$ for $r+1 \leq j \leq r+m-1.$ From Lemma \ref{lem idem relations}, we see then that for all $\tau \in \SYT(\lambda^{(r+m)}),$ $\epsilon_{r}^{(r+m)}(e_{\tau}) = \alpha_r(\tau) v_{\mu,\beta}$ for some unique $(\mu,\beta) \in C_{\lambda,r}$ and scalar $\alpha_r(\tau).$ In particular, if $s_{j}(\tau) > \tau$ for $r+1 \leq j \leq r+m-1$, then using the relation $\epsilon_r^{(r+m)}T_j = \epsilon_r^{(r+m)},$ we find that 
    $$a_r(s_{j}(\tau)) = t\frac{1-t^{c_{\tau}(j+1)-c_{\tau}(j)-1}}{1-t^{c_{\tau}(j+1)-c_{\tau}(j)}} a(\tau).$$ Lastly, since $a_r(\gamma_{\mu,\beta}) = 1$ for all $\mu \in C_{\lambda,r}$, by induction we see that $|a(\tau)| \leq 1.$
\end{proof}

The action of the partially-trivial idempotents $\epsilon_{r}^{(n)}$ is stable in the following sense.

\begin{lem}\label{lem stability of partial triv idem}
    If $\tau \in \SYT_{\infty}(\lambda)$ and $r \geq n_{\lambda},$ then for all $m' \geq m \geq \max\{\lambda_1, \mathrm{rk}(\tau)\},$ $\epsilon_{r}^{(r+m)}(e_{\tau}) = \epsilon_{r}^{(r+m')}(e_{\tau}).$
\end{lem}
\begin{proof}
    This is a consequence of the fact that from Lemma \ref{lem action of partial triv idem}, the right-hand side of the equations $\epsilon_{r}^{(r+m)}(e_{\tau}) = \alpha_r(\tau) v_{\mu,\beta}$ do not depend on $m$ so long as $m \geq \max\{\lambda_1, \mathrm{rk}(\tau)\}.$
\end{proof}

We are now ready to give the definition of the modules $\hat{S}_{\lambda}.$

\begin{defn}\label{defn inf specht module closure}
    For $\lambda \in \Par$, define  $\hat{S}_{\lambda}$ as the space of formal sums $\sum_{\tau \in \SYT_{\infty}(\lambda)} c_{\tau} e_{\tau}$ such that for all $c > 0$ there exist only finitely many $\tau \in \SYT_{\infty}(\lambda)$ such that $|c_{\tau}| \geq c$. We give $\hat{S}_{\lambda}$ the norm $||\sum_{\tau \in \SYT_{\infty}(\lambda)} c_{\tau} e_{\tau} ||:= \sup_{\tau \in \SYT_{\infty}(\lambda)} |c_{\tau}|.$ In other words, $\hat{S}_{\lambda}:= c_{0}(\SYT_{\infty}(\lambda))$. For $r \geq 1$, define the operator $\epsilon_r: \hat{S}_{\lambda} \rightarrow \hat{S}_{\lambda}$ as 
    $$\epsilon_{r}\left( \sum_{\tau \in \SYT_{\infty}(\lambda)} c_{\tau} e_{\tau}\right):= \sum_{\tau \in \SYT_{\infty}(\lambda) } c_{\tau} \lim_{m \rightarrow \infty }\epsilon_{r}^{(r+ m)}(e_{\tau}).$$ For $i \geq 1,$ define the operator $T_i: \hat{S}_{\lambda} \rightarrow \hat{S}_{\lambda}$ by 
    \begin{align*}
        T_i&\left( \sum_{\tau \in \SYT_{\infty}(\lambda)} c_{\tau} e_{\tau} \right):= \\
        & \sum_{\{i,i+1 \} ~ \text{same row of} ~ \tau } c_{\tau}e_{\tau} + \sum_{\{i,i+1 \} ~ \text{same column of}~ \tau} (-t)c_{\tau}e_{\tau}+ \sum_{\tau< s_i(\tau)}  \left( e_{s_{i}(\tau)} + \frac{1-t}{1-t^{c_{\tau}(i+1)-c_{\tau}(i)}} e_{\tau} \right)\\
        &+ \sum_{\tau > s_{i}(\lambda)}\left( t\frac{(1-t^{c_{\tau}(i) - c_{\tau}(i+1) -1})(1-t^{c_{\tau}(i)-c_{\tau}(i+1)+1})}{(1-t^{c_{\tau}(i)-c_{\tau}(i+1)})^2} e_{s_{i}(\tau)} - t^{c_{\tau}(i)-c_{\tau}(i+1)} \frac{1-t}{1-t^{c_{\tau}(i)-c_{\tau}(i+1)}}e_{\tau} \right).\\
    \end{align*}
\end{defn}

Crucially, note that $S_{\lambda^{(\infty)}}$ is a dense subspace of $\hat{S}_{\lambda}.$ We use this fact repeatedly.

\begin{lem}\label{lem operators are bounded}
    The operators $T_i,\epsilon_r: \hat{S}_{\lambda} \rightarrow \hat{S}_{\lambda}$ are bounded with $||T_i||,||\epsilon_r|| \leq 1.$ In particular, $||T_i|| = 1$, and $||\epsilon_r|| = 1$ for $r \geq n_{\lambda}.$
\end{lem}
\begin{proof}
    Remark \ref{remark explicit action} shows directly that $||T_i|| \leq 1.$ To see that $||T_i||=1$, note that for all $i \geq 1$ we may find $\tau \in \SYT_{\infty}(\lambda)$ such that $i,i+1$ appear in the same row of $\tau.$ Namely, define $\tau$ by labeling the first row of $\lambda^{(\infty)}$ first with the labels $1,\dots,i,i+1$ and place the entries $\{i+2,\dots,i+|\lambda|+2\}$ in the row-standard orientation in $\lambda \subset \lambda^{(\infty)}$, and filling the remaining entries of $\lambda^{(\infty)}/\lambda^{(i+|\lambda|+2)}$ consecutively. Then $T_i(e_{\tau}) = e_{\tau}$ so indeed $||T_i|| = 1.$ The bound $||\epsilon_r|| \leq 1$ for all $r \geq 0$ follows from Lemma \ref{lem action of partial triv idem} and the simple fact that for all $r \geq n_{\lambda}$ and $(\mu,\beta) \in C_{\lambda,r}$, $||v_{\mu,\beta}|| = ||\epsilon_{r}^{(r+\lambda_1)}(e_{\gamma_{\mu,\beta}})|| \leq ||\epsilon_{r}^{(r+\lambda_1)}|| \cdot ||e_{\gamma_{\mu,\beta}}|| \leq 1$. To see that $||\epsilon_r|| = 1$ for $r \geq n_{\lambda}$ simply note that $\epsilon_r(e_{\tau_{\lambda}}) = e_{\tau_{\lambda}}.$
\end{proof} 

\begin{prop}\label{prop inf specht are alm-sym}
    The operators $T_i,\epsilon_r: \hat{S}_{\lambda} \rightarrow \hat{S}_{\lambda}$ define an almost-symmetric representation of $\widehat{\cH}_{\infty}.$
\end{prop}
\begin{proof}
    It suffices to check that the operators $T_i,\epsilon_r$ define a representation of $\cH_{\infty}$ on $\hat{S}_{\lambda}.$ However, it is clear from Definition \ref{defn inf specht module closure} by restricting to $S_{\lambda^{(\infty)}}$, that the operators $T_i$,$\epsilon_{r}$ satisfy the relations of $\widehat{\cH}_{\infty}$ on $S_{\lambda^{(\infty)}}$. Thus, from the norm bounds in Lemma \ref{lem operators are bounded} we see that they indeed define an action of $\widehat{\cH}_{\infty}$ which is almost-symmetric since $\epsilon_{n_{\lambda}} \neq 0$.
\end{proof}

We now show that the $\hat{S}_{\lambda}$ are irreducible representations of $\widehat{\cH}_{\infty}.$ First, we need the following technical result which guarantees that no non-zero vector in $\hat{S}_{\lambda}$ can be annihilated by every $\epsilon_r$ operator for $r \geq n_{\lambda}.$ We use an analytic approach to greatly simplify the argument.

\begin{lem}\label{lem intersection of ker 0}
   For all $\lambda \in \Par$, as operators on $\hat{S}_{\lambda},$
   $$\bigcap_{r \geq n_{\lambda}} \mathrm{Ker} (\epsilon_r) = 0.$$
\end{lem}
\begin{proof}
    Let $v \in \bigcap_{r \geq n_{\lambda}} \mathrm{Ker} (\epsilon_r)$ and $c>0.$ Since $S_{\lambda^{(\infty)}}$ is dense in $\hat{S}_{\lambda},$ we may find some $v_0 \in S_{\lambda^{(\infty)}}$ such that $||v-v_0|| < c.$ Let $r \geq n_{\lambda}$ be such that $v_0 \in S_{\lambda^{(r)}}$. Then $\epsilon_r(v_0) = v_0.$ Now we see that 
    $$||v_0|| = ||v_0-0|| = ||\epsilon_r(v_0) - \epsilon_r(v)|| \leq ||\epsilon_r||\cdot ||v_0-v|| = ||v_0-v|| < c.$$ Thus 
    $$||v|| = ||v-v_0 + v_0|| \leq \max\{ ||v-v_0||, ||v_0|| \} < c.$$ Since $||v||< c$ for all $c > 0$, $||v|| = 0$ and so $v=0.$
\end{proof}

Using a combination of combinatorial and analytic arguments, it is not hard to show that the $\widehat{\cH}_{\infty}$ modules $\hat{S}_{\lambda}$ are irreducible.

\begin{prop}\label{prop inf specht are irred}
    For all $\lambda \in \Par,$ $\hat{S}_{\lambda}$ is an irreducible $\widehat{\cH}_{\infty}$-module.
\end{prop}
\begin{proof}
   Set $n_{\lambda}:= |\lambda| + \lambda_1$. Let $v \in \hat{S}_{\lambda} \setminus \{0\}.$ Then by Lemma \ref{lem intersection of ker 0} there exists $r \geq n_{\lambda}$ such that $\epsilon_r(v) \neq 0.$ By Lemma \ref{lem action of partial triv idem}, $\epsilon_r(v) \in S_{\lambda^{(r+\lambda_1)}}.$  Since $S_{\lambda^{(r+\lambda_1)}}$ is an irreducible $\cH_{r+\lambda_1}$-module we can find some $X \in \cH_{r+\lambda_1}$ such that $X \epsilon_r(v) = e_{\tau_{\lambda}}$ where $\tau_{\lambda}$ is the column standard element of $\SYT_{\infty}(\lambda).$ We know that $\cH_{\infty}.e_{\tau_{\lambda}} = S_{\lambda^{(\infty)}}$ is norm dense in $\hat{S}_{\lambda}.$ Therefore, $\widehat{\cH}_{\infty}.v = \widehat{\cH}_{\infty}.e_{\tau_{\lambda}}$ is also norm dense in $\hat{S}_{\lambda}.$ Thus if $V \subseteq \hat{S}_{\lambda}$ is a nonzero closed $\widehat{\cH}_{\infty}$-submodule of $\hat{S}_{\lambda}$, then necessarily $V = \hat{S}_{\lambda}$, i.e., $\hat{S}_{\lambda}$ is irreducible.
\end{proof}

\subsection{Partial Classification of Irreducible Representations}

In this section, we prove a partial classification of the almost-symmetric representations of $\widehat{\cH}_{\infty}.$ Namely, we prove that every irreducible almost-symmetric $\widehat{\cH}_{\infty}$-module contains some $\hat{S}_{\lambda}$ as a dense submodule. We roughly use the following argument. First, show that the presence of non-zero partially-invariant vectors in $\cH_{\infty}$-modules guarantees an entire copy of some module $S_{\lambda^{(\infty)}}.$ Then, using an analytic argument, we argue that if an irreducible almost-symmetric $\widehat{\cH}_{\infty}$-module contains a full copy of some $S_{\lambda^{(\infty)}}$, then it in fact admits a non-zero bounded module map $\hat{S}_{\lambda} \rightarrow V$ with dense image. This gives a partial classification of the irreducible almost-symmetric $\widehat{\cH}_{\infty}$-modules.

We start this section with a two combinatorial lemmas.

\begin{lem}\label{lem dominance criterion}
    Let $k,j \geq 0$. Suppose $V$ is a $\cH_{k+j}$-module and $v \in V \setminus \{0\}$ is $T_i$-invariant for $k+1 \leq i \leq k+j-1.$ Then any $S_{\lambda}$ which admits a $\cH_{k+j}$-module injection into $\cH_{k+j}.v$ satisfies $(j,1^k) \trianglelefteq \lambda.$
\end{lem}
\begin{proof}
Let $U:= \cH_{k+j}.v . $ Since $v$ is $T_i$-invariant for $k+1 \leq i \leq k+j-1$, by Frobenius reciprocity there exists a $\cH_{k+j}$-module surjection $\mathrm{Ind}_{\cH_{1^k,j}}^{\cH_{k+j}} \mathrm{triv}(1^k,j) \rightarrow U.$ Using the branching rules from Remark \ref{remark branching}, we see that 
$$\mathrm{Ind}_{\cH_{1^k,j}}^{\cH_{k+j}} \mathrm{triv}(1^k,j)  \cong \bigoplus_{\lambda \vdash k+j} K_{\lambda,(j,1^k)} S_{\lambda}$$ where $K_{\lambda,\mu}$ are the \textbf{Kostka numbers}. Thus if $K_{\lambda,(j,1^k)} \neq 0$, then $\lambda \trianglerighteq (j,1^k).$
\end{proof}

\begin{lem}\label{lem rep stability for inf hecke}
    Let $k \geq 0.$ Suppose $\mu_{k} \subset \mu_{k+1} \subset \dots$ is an increasing sequence of partitions such that for all $j \geq 0$, $|\mu_{k+j}| = k+j$ and $\mu_{k+j} \trianglerighteq (j,1^k).$ Then there exists a partition $\lambda$ with $|\lambda| \leq k$ such that for all sufficiently large $j,$ $\lambda^{(k+j)} = \mu_{k+j}.$
\end{lem}
\begin{proof}
    Since $(\mu_{2k})_1 \geq k$ we may choose $\lambda$ to be the unique partition $\lambda$ with $|\lambda| \leq k$ such that $\mu_{2k} = \lambda^{(2k)}.$ The condition $\mu_{k+j} \trianglerighteq (j,1^k)$ tells us that $(\mu_{k+j})_1 \geq j.$ This means that there are at most $k$ boxes of every $\mu_{k+j}$ not in the first row of the diagram. By the pigeonhole principle, as the chain of partitions $\mu_{2k} \subset \mu_{2k+1}\subset \dots$ grows, there must be only finitely many boxes which are added to rows which are not the first row. Therefore, there exists some partition $\lambda'$ with $|\lambda'| \leq k$ such that for all sufficiently large $j$, $\mu_{k+j} = (\lambda')^{(k+j)}.$
\end{proof}

The above combinatorial lemmas have the following representation-theoretic consequence.

\begin{prop}\label{prop cyclic vect give specht mods}
    If $V$ is a $\cH_{\infty}$-module and $v \in V \setminus \{0\}$ is $T_i$-invariant for all $i> k$, then the cyclic submodule $\cH_{\infty}.v \subseteq V$ contains a submodule isomorphic to $S_{\lambda^{(\infty)}}$ for some partition $\lambda$ with $|\lambda| \leq k.$
\end{prop}
\begin{proof}
    For $j \geq 0$ define $U_{k+j}:= \cH_{k+j}.v$ which form a filtration 
    $$U_{k} \subseteq U_{k+1} \subseteq \dots.$$ By decomposing each $U_{k+j}$ into irreducible $\cH_{k+j}$-submodules, we can find a sequence of partitions $(\mu_{k+j})_{j \geq 0}$ such that 
    \begin{itemize}
        \item for all $j \geq 0$, $|\mu_{k+j}| = k+j$
        \item $\mu_{k} \subset \mu_{k+1}\subset \dots$
        \item $S_{\mu_k} \subseteq S_{\mu_{k+1}}\subseteq \dots$ forms a sub-filtration of $U_{k} \subseteq U_{k+1} \subseteq \dots.$
    \end{itemize}

    By Lemma \ref{lem dominance criterion}, we know that for all $j \geq 0,$ $\mu_{k+j} \trianglerighteq (j,1^k).$ Thus by Lemma \ref{lem rep stability for inf hecke}, it follows that there exists some $\lambda$ with $|\lambda| \leq k$ such that for all $j \gg 0,$ $\mu_{k+j} = \lambda^{(k+j)}$. Thus we have a constructed a $\cH_{\infty}$-submodule of $\cH_{\infty}.v$ of the form 
    $ S_{\lambda^{(\infty)}} = \bigcup_{j \gg 0} S_{\lambda^{(k+j)}}.$
\end{proof}

We require the next technical analytic lemma.

\begin{lem}\label{lem extension of reps}
    Let $V$ be an almost-symmetric $\widehat{\cH}_{\infty}$-module and suppose $f: S_{\lambda^{(\infty)}} \rightarrow V$ is an injective $\cH_{\infty}$-module map. Then f extends uniquely to a non-zero bounded $\widehat{\cH}_{\infty}$-module map $\hat{f}: \hat{S}_{\lambda} \rightarrow V.$
\end{lem}
\begin{proof}
    Since $V$ is a $\widehat{\cH}_{\infty}$-module, we know that $||T_i||_{V} \leq 1$ for all $i \geq 1.$ Since $S_{\lambda^{(\infty)}}$ is cyclic as a $\cH_{\infty}$-module we know that necessarily $f: S_{\lambda^{(\infty)}}\rightarrow V$ is bounded with respect to the norm topology on $S_{\lambda^{(\infty)}}$ considered as a subspace of $\hat{S}_{\lambda}$. Therefore, $f$ extends uniquely to a non-zero bounded $\widehat{\cH}_{\infty}$-module map $\hat{f}:\hat{S}_{\lambda} \rightarrow V.$
\end{proof}

Combining our results so far shows that every almost-symmetric representation admits a map from one of the modules $\hat{S}_{\lambda}.$

\begin{cor}\label{cor extension of reps}
    If $V$ is an almost-symmetric $\widehat{\cH}_{\infty}$-module, then there exists a partition $\lambda$ with a nonzero $\widehat{\cH}_{\infty}$-module map $\hat{S}_{\lambda} \rightarrow V.$ 
\end{cor}
\begin{proof}
    Since $V$ is an almost-symmetric $\widehat{\cH}_{\infty}$-module, there exists some $k \geq 0$ such that $\epsilon_k \neq 0$ as an operator on $V.$ Then there exists some $v \in V \setminus \{0\}$ such that $\epsilon_k(v) \neq 0.$ Note that for all $i > k$, 
    $T_i \epsilon_k(v) = (T_i\epsilon_k)(v) = \epsilon_k(v)$. By Proposition \ref{prop cyclic vect give specht mods}, we may find a partition $\lambda$ with $|\lambda| \leq k$ such that $S_{\lambda^{(\infty)}}$ embeds as a $\cH_{\infty}$-submodule into $\cH_{\infty}.v \subseteq V.$ By Lemma \ref{lem extension of reps}, there exists a unique extension of the embedding $S_{\lambda^{(\infty)}} \rightarrow V$ to a bounded $\widehat{\cH}_{\infty}$-module map $\hat{S}_{\lambda} \rightarrow V$. 
\end{proof} 

Finally, we are able to complete our partial classification of the irreducible almost-symmetric representations of $\widehat{\cH}_{\infty}.$

\begin{thm}\label{thm classification of irreps}
    If $V$ is an irreducible almost-symmetric $\widehat{\cH}_{\infty}$-module, then there exists a bounded $\widehat{\cH}_{\infty}$-module embedding with dense image $ \hat{S}_{\lambda}\rightarrow  V$ for some $\lambda \in \Par.$
\end{thm}
\begin{proof}
    Suppose $V$ is an irreducible almost-symmetric $\widehat{\cH}_{\infty}$-module. By Corollary \ref{cor extension of reps}, there exists a partition $\lambda$ along with a nonzero bounded $\widehat{\cH}_{\infty}$-module map $\hat{S}_{\lambda} \rightarrow V$. By Proposition \ref{prop inf specht are irred}, we know that $\hat{S}_{\lambda}$ is irreducible so the map $\hat{S}_{\lambda} \rightarrow V$ is a bounded $\widehat{\cH}_{\infty}$-module linear isomorphism. In particular, the map $\hat{S}_{\lambda} \rightarrow V$ must be injective as otherwise, by continuity, the kernel would be a closed $\widehat{\cH}_{\infty}$-invariant subspace of $\hat{S}_{\lambda}.$ Finally, the image of the embedding $\hat{S}_{\lambda} \rightarrow V$ must have dense image in $V$ as otherwise the closure of the image would be a closed proper $\widehat{\cH}_{\infty}$-invariant subspace of $V$.
\end{proof}

\section{Regularized Traces of Irreducible Almost-symmetric Representations}

The traces values of irreducible representations of Hecke algebras have an interesting place in algebraic combinatorics relating to chromatic quasi-symmetric functions, immanants, LLT polynomials, and Kazhdan--Lusztig theory \cite{Haiman} \cite{SW16} \cite{MSW}. Since we have found a natural analogue for at least some of the irreducible representations of $\cH_{\infty}$, it seems natural to try and extend this theory to the infinite setting. 

In the finite dimensional setting, there is a canonically well-defined trace functional for irreducible Hecke algebra representations. Even in the infinite dimensional setting, for the infinite symmetric group and infinite Hecke algebras over $\mathbb{C}$ there are methods for canonically defining traces. However, in the infinite dimensional setting over the non-archimedean field $\bF$ we run into some difficulties. First, any attempt at defining limits of normalized trace functionals in the way that is usually done over $\mathbb{C}$ for the representations $S_{\lambda^{(n)}}$ fails over $\bF$. That leaves the only available option to instead consider traces in the functional analysis sense of trace-class operators on Hilbert spaces. Here again, we run into difficulties since the Banach spaces $\hat{S}_{\lambda}$ are not reflexive (at least for $\lambda \neq \emptyset).$ Our solution is to consider inner products in the sense of Definition \ref{defn inner products}. Typically, to compute the trace of an operator $x$ with respect to an inner product, one considers sums of the form $\langle x(v_i),v_i \rangle$ ranging over an orthonormal basis. In our case, most elements of $\cH_{\infty}$ have non-convergent trace values under this definition. To arrive at well-defined trace values for all elements of $\cH_{\infty}$ on $\hat{S}_{\lambda}$, we instead re-normalize our basis and use the same type of series but for the rescaled basis. 

Our choice of rescaling is later justified in Theorems \ref{thm integrality} and \ref{rationality thm} where we prove the surprising result that the convergent series $\Gamma_{\lambda}(T_{\sigma})$ for all $\sigma \in \mathfrak{S}_{\infty}$ are integral (in the non-archimedean sense) and rational. An arbitrary choice for this normalization would not satisfy these properties.

\subsection{Inner Products on Irreducible Representations}

In this section, we show that the representations $\hat{S}_{\lambda}$ admit (essentially) unique inner products $(-,-)_{\lambda}$ subject to the constraint the operators $T_i$ must be self-adjoint. In order to build these inner products on the full spaces $\hat{S}_{\lambda}$, we first look at the finite dimensional representations $S_{\lambda}.$

\begin{prop}
    Let $\lambda \in \Par$ with $|\lambda| = n.$ There exists a unique inner product $\langle-,-\rangle_{\lambda}: S_{\lambda} \times S_{\lambda} \rightarrow \bF$ such that for all $1\leq i \leq n-1, $ $T_i^* = T_i$ and $\langle e_{\tau_{\lambda}}, e_{\tau_{\lambda}} \rangle_{\lambda} = 1.$ 
\end{prop}
\begin{proof}
    Suppose $\langle - , - \rangle:S_{\lambda} \times S_{\lambda} \rightarrow \bF$ is some symmetric bilinear form such that for all $1\leq i \leq n-1$ and $v,w \in S_{\lambda},$ $\langle T_i(v),w\rangle = \langle v, T_{i}(w) \rangle.$ It is easy to see that for all $1\leq j \leq n$ and $1\leq i \leq n-1,$ and for all $v, w\in S_{\lambda}, $ $\langle \theta_j(v), w\rangle = \langle v, \theta_j(w) \rangle$ and $\langle \varphi_j(v), w\rangle = - \langle v, \varphi_j(w) \rangle.$ Recall that if $\tau_1, \tau_2 \in \SYT(\lambda)$ with $\tau_1 \neq \tau_2$ then there exists some $1\leq j \leq n$ such that $c_{\tau_1}(j) \neq c_{\tau_2}(j).$ Thus for any $\tau_1 \neq \tau_2$ we can find $1\leq j \leq n$ such that $t^{c_{\tau_1}(j)} \neq t^{c_{\tau_2}(j)}$ so that 
    $$(t^{c_{\tau_1}(j)} - t^{c_{\tau_2}(j)}) \langle e_{\tau_1},e_{\tau_2}\rangle = \langle t^{c_{\tau_1}(j)}e_{\tau_1},e_{\tau_2}\rangle - \langle e_{\tau_1},t^{c_{\tau_2}(j)}e_{\tau_2}\rangle = \langle \theta_j(e_{\tau_1}),e_{\tau_2}\rangle - \langle e_{\tau_1},\theta_j(e_{\tau_2})\rangle = 0$$ which implies that $\langle e_{\tau_1},e_{\tau_2}\rangle = 0.$ Therefore, the set $\{e_{\tau}| \tau \in \SYT(\lambda) \}$ is an orthogonal basis for $S_{\lambda}$ with respect to $\langle - , - \rangle$. Thus any such inner product $\langle-,-\rangle$ on $S_{\lambda}$ is uniquely determined by the values $\langle e_{\tau}, e_{\tau} \rangle\in \bF$ for $\tau \in \SYT(\lambda).$ Furthermore, if $\tau \in \SYT(\lambda)$ and $1\leq i \leq n-1$ with $s_i(\tau) > \tau,$ then on one hand,
    $$\langle \varphi_i(e_{\tau}), \varphi_i(e_{\tau}) \rangle = (t^{c_{\tau}(i)}-t^{c_{\tau}(i+1)})^2\langle e_{s_i(\tau)}, e_{s_i(\tau)} \rangle.$$ 
    Using Lemma \ref{lem adjoint properties} we know that for all $1\leq i \leq n-1$, since $T_i$ and $\theta_{i}$ are self-adjoint with respect to $\langle -,- \rangle$, $\varphi_i = tT_{i}^{-1}\theta_i - \theta_i tT_{i}^{-1}$ is skew-adjoint with respect to each $\langle -,- \rangle.$ On the other hand, 
    \begin{align*}
        &\langle \varphi_i(e_{\tau}), \varphi_i(e_{\tau}) \rangle\\
        &= \langle e_{\tau}, -\varphi_i^2(e_{\tau}) \rangle\\
        &= \langle e_{\tau}, -(t\theta_{i}-\theta_{i+1})(t\theta_{i+1}-\theta_{i})(e_{\tau}) \rangle \\
        &= -(t^{c_{\tau}(i)+1}-t^{c_{\tau}(i+1)})(t^{c_{\tau}(i+1)+1}-t^{c_{\tau}(i)})\langle e_{\tau}, e_{\tau} \rangle.\\ 
    \end{align*}
By assumption, $c_{\tau}(i+1) \notin \{c_{\tau}(i)-1 ,c_{\tau}(i), c_{\tau}(i)+1\}$ so because $t$ is a pseudo-uniformizer in $\bF,$ $t^{c_{\tau}(i+1)}  \notin \{ t^{c_{\tau}(i)-1}, t^{c_{\tau}(i)}, t^{c_{\tau}(i)+1} \}.$ Therefore, 
$$\langle e_{s_{i}(\tau)},e_{s_{i}(\tau)} \rangle = \frac{(t^{c_{\tau}(i+1)- c_{\tau}(i)}-t)(t^{c_{\tau}(i+1)- c_{\tau}(i)+1}-1)}{(t^{c_{\tau}(i+1)- c_{\tau}(i)}-1)^2}\langle e_{\tau},e_{\tau} \rangle.$$ Thus if $\langle e_{\tau},e_{\tau} \rangle \neq 0$, then $\langle e_{s_{i}(\tau)},e_{s_{i}(\tau)} \rangle \neq 0.$ Since $\tau_{\lambda}$ is the unique minimal element of $\SYT(\lambda)$ if $\langle e_{\tau_{\lambda}}, e_{\tau_{\lambda}} \rangle  \neq 0$, then the values $\langle e_{\tau}, e_{\tau} \rangle$  for $\tau \in \SYT(\lambda)$ are uniquely determined and nonzero. In particular, we may choose $\langle e_{\tau_{\lambda}}, e_{\tau_{\lambda}}\rangle = 1$ and define the other values $\langle e_{\tau}, e_{\tau}\rangle$ accordingly. It follows that there exists a unique inner product $\langle - , -\rangle_{\lambda}$ satisfying the criteria in the statement of the proposition. 
\end{proof}

To better keep track of the outputs of the inner products $\langle-,-\rangle_{\lambda}$, we use the following recursive definition. 

\begin{defn}
    Define the function $a: \SYT_{\infty}(\lambda) \rightarrow \bF$ recursively by 
    \begin{itemize}
        \item $a(\tau_{\lambda}):= 1$
        \item if $s_i(\tau) > \tau$, then 
        $a(s_i(\tau)):= \frac{(t^{c_{\tau}(i+1)- c_{\tau}(i)}-t)(t^{c_{\tau}(i+1)- c_{\tau}(i)+1}-1)}{(t^{c_{\tau}(i+1)- c_{\tau}(i)}-1)^2} a(\tau).$
    \end{itemize}
    Define the \textbf{inversion number} $\mathrm{inv}(\tau)$ of $\tau \in \SYT(\lambda)$ to be the minimal number $r$ such that there exist $i_1,\dots,i_r \geq 1$ with $\tau = s_{i_r}\cdots s_{i_1}(\tau_{\lambda}).$ 
\end{defn}

The values $a(\tau)$ are complicated, but luckily they have predictable norms $|a(\tau)|.$

\begin{lem}\label{lem bounds on norms}
    For all $\tau \in \SYT_{\infty}(\lambda),$ $|a(\tau)| = |t|^{\mathrm{inv}(\tau)}.$ In particular, $|a(\tau)| \leq 1$ for all $\tau \in \SYT_{\infty}(\lambda)$.
\end{lem}
\begin{proof}
    If $\tau \in \SYT(\lambda)$ and $i \geq 1,$ then $c_{\tau}(i+1) - c_{\tau}(i) \geq 2.$ Therefore, by definition 
    \begin{align*}
        |a(s_{i}(\tau))| &= \left| \frac{(t^{c_{\tau}(i+1)- c_{\tau}(i)}-t)(t^{c_{\tau}(i+1)- c_{\tau}(i)+1}-1)}{(t^{c_{\tau}(i+1)- c_{\tau}(i)}-1)^2} a(\tau) \right| 
        \\
        &= |t| \cdot \left| \frac{(t^{c_{\tau}(i+1)- c_{\tau}(i)-1}-1)(t^{c_{\tau}(i+1)- c_{\tau}(i)+1}-1)}{(t^{c_{\tau}(i+1)- c_{\tau}(i)}-1)^2} \right| \cdot |a(\tau)| \\
        &= |t| \cdot |a(\tau)| \\
    \end{align*}
   Therefore, since $a(\tau_{\lambda}) = 1$, by induction, we see that $|a(\tau)| = |t|^r$ where $r$ is the minimal number such that there exist $i_1,\dots,i_r \geq 1$ with $\tau = s_{i_r}\cdots s_{i_1}(\tau_{\lambda}).$ In other words, 
   $|a(\tau)| = |t|^{\mathrm{inv}(\tau)}.$
\end{proof}

\begin{remark}
    Equivalently, $\mathrm{inv}(\tau)$ is the number of inversion pairs of $\tau$, i.e., boxes $(\square_1,\square_2)$ of the diagram $\lambda^{(\infty)}$ such that $\tau_{\lambda}(\square_1) < \tau_{\lambda}(\square_2)$ and $\tau(\square_1) > \tau(\square_2).$ This may be easily verified by induction along the poset $\SYT_{\infty}(\lambda).$ Similarly, in the same way one may verify the following explicit formula:

    $$a(\tau) = t^{\mathrm{inv(\tau)}} \prod_{(\square_1,\square_2) \in \mathrm{Inv}(\tau)}\frac{(1-t^{c(\square_2)-c(\square_1)-1})(1-t^{c(\square_2)-c(\square_1)+1})}{(1-t^{c(\square_2)-c(\square_1)})^2}$$
    where $\mathrm{Inv}(\tau)$ is the set of inverse pairs of $\tau.$
\end{remark}

\begin{example}
   Consider $\tau \in \SYT(3,2,1)$ given by $\tau = \begin{ytableau}
       1 & 2 & 3 \\
       4 & 6 & \none\\
       5 & \none & \none \\
   \end{ytableau}.$ If we label the boxes of $\tau$ by their labels, then the inversion pairs of $\tau$ are $(4,2),(4,3),(5,2),(5,3),$ and $(6,3)$ which have content gaps $c_{\tau}(\square_2)-c_{\tau}(\square_1)$ given by $2,3,3,4,$ and $2$ respectively. 
   Therefore, 
   $$a(\tau) = t^{5} \left( \frac{(1-t)(1-t^3)}{(1-t^2)^2} \right)^2 \left( \frac{(1-t^2)(1-t^4)}{(1-t^3)^2} \right)^2 \left( \frac{(1-t^3)(1-t^5)}{(1-t^4)^2} \right) = t^5 \frac{(1-t)^2(1-t^3)(1-t^5)}{(1-t^2)^2}.$$
\end{example}

Using the values $a(\tau)$, we are able to give an explicit description of the symmetric bilinear forms $\langle-.-\rangle_{\lambda^{(\infty)}}$ extending those forms $\langle-.-\rangle_{\lambda^{(n)}}$ on the finite dimensional representations $S_{\lambda^{(n)}}.$

\begin{defn}
    Define the symmetric bilinear form $\langle - , - \rangle_{\lambda^{(\infty)}}: S_{\lambda^{(\infty)}} \times S_{\lambda^{(\infty)}} \rightarrow \bF$ by 
    $$\left\langle \sum_{\tau \in \SYT_{\infty}(\lambda)} \alpha_{\tau} e_{\tau},  \sum_{\tau \in \SYT_{\infty}(\lambda)} \beta_{\tau} e_{\tau}\right\rangle_{\lambda^{(\infty)}}:= \sum_{\tau \in \SYT_{\infty}(\lambda)} \alpha_{\tau} \beta_{\tau} a(\tau).$$
\end{defn}

The forms $\langle-,-\rangle_{\lambda^{(\infty)}}$ are continuous with respect to the norm topology on $S_{\lambda^{(\infty)}} \subseteq \hat{S}_{\lambda}.$

\begin{lem}\label{lem inner product is continuous}
    For all $v,w \in S_{\lambda^{(\infty)}}$, $|\langle v,w \rangle_{\lambda^{(\infty)}}| \leq ||v||\cdot ||w|| .$
\end{lem}
\begin{proof}
    Let $v = \sum_{\tau \in \SYT_{\infty}(\lambda)} \alpha_{\tau} e_{\tau}, w = \sum_{\tau \in \SYT_{\infty}(\lambda)} \beta_{\tau} e_{\tau} \in S_{\lambda^{(\infty)}}.$ Then
    $$\left|\left\langle \sum_{\tau \in \SYT_{\infty}(\lambda)} \alpha_{\tau} e_{\tau},  \sum_{\tau \in \SYT_{\infty}(\lambda)} \beta_{\tau} e_{\tau}\right\rangle_{\lambda^{(\infty)}} \right|= \left|\sum_{\tau \in \SYT_{\infty}(\lambda)} \alpha_{\tau} \beta_{\tau} a(\tau)\right| \leq \sup_{\tau \in \SYT_{\infty}(\lambda)} | \alpha_{\tau} \beta_{\tau} a(\tau)|$$ so from Lemma \ref{lem bounds on norms} we find that 
    $$|\langle v, w \rangle_{\lambda^{(\infty)}}| \leq \sup_{\tau \in \SYT_{\infty}(\lambda)} | \alpha_{\tau} \beta_{\tau}| \leq \sup_{\tau \in \SYT_{\infty}(\lambda)} | \alpha_{\tau}| \cdot \sup_{\tau \in \SYT_{\infty}(\lambda)} |  \beta_{\tau}| = ||v|| \cdot ||w||.$$
\end{proof}

By continuity, we are able to define the desired inner products on the spaces $\hat{S}_{\lambda}.$

\begin{thm}\label{thm unique inner product}
    There exists a unique inner product $( -,-)_{\lambda}: \hat{S}_{\lambda} \times \hat{S}_{\lambda} \rightarrow \bF$ such that $( e_{\tau_{\lambda}}, e_{\tau_{\lambda}} )_{\lambda} = 1$ and for all $i \geq 1,$ $T_i^* = T_i.$
\end{thm}
\begin{proof}
    By Lemma \ref{lem inner product is continuous}, we find that $\langle -, - \rangle_{\lambda^{(\infty)}}:S_{\lambda^{(\infty)}}\times S_{\lambda^{(\infty)}} \rightarrow \bF$ is continuous. Therefore, as $S_{\lambda^{(\infty)}}$ is norm dense in $\hat{S}_{\lambda}$, there exists a unique extension of $\langle -, - \rangle_{\lambda^{(\infty)}}$ to $\hat{S}_{\lambda}\times \hat{S}_{\lambda}$. Call this map $( -,-)_{\lambda}.$ It is easy to see that $( -,-)_{\lambda}$ satisfies the requirements of the theorem statement. In particular, $(-,-)_{\lambda}$ has the explicit description 
    $$\left( \sum_{\tau \in \SYT_{\infty}(\lambda)} \alpha_{\tau} e_{\tau},  \sum_{\tau \in \SYT_{\infty}(\lambda)} \beta_{\tau} e_{\tau} \right)_{\lambda} = \sum_{\tau \in \SYT_{\infty}(\lambda)} \alpha_{\tau} \beta_{\tau} a(\tau).$$ 
    Furthermore, $( -,-)_{\lambda}$ is unique since any other $( -,-)_{\lambda}'$ satisfying the requirements of the theorem statement would necessarily agree with $( -,-)_{\lambda}$ on the dense subspace $S_{\lambda^{(\infty)}}$ so by continuity $( -,-)_{\lambda}= ( -,-)_{\lambda}'.$
\end{proof}

\subsection{Regularized Traces of Irreducible Representations}

In the absence of the usual tools of analysis over $\mathbb{C}$ that would allow one to consider traces of operators acting on the spaces $\hat{S}_{\lambda},$ we instead take a combinatorial approach to defining analogous regularized trace functionals. In order to match conventions with the author's previous work \cite{BWMurnaghan} (see Remark \ref{remark vector-valued} and Theorem \ref{rationality thm} for more details), we need to re-normalize the bases $\{ e_{\tau}|\tau \in \SYT_{\infty}(\lambda)\}$ for each $\hat{S}_{\lambda}.$

\begin{defn}\label{normalization defn}
    For $\lambda \in \Par$ and $\tau \in \SYT_{\infty}(\lambda),$ define 
    $$\tilde{e}_{\tau}:= \prod_{(\square_1,\square_2) \in \mathrm{Inv}(\tau)} \left(\frac{1-t^{c(\square_2)-c(\square_1)}}{1-t^{c(\square_2)-c(\square_1)+1}} \right) e_{\tau}.$$
\end{defn}

Note that $||\tilde{e}_{\tau}|| = ||e_{\tau}|| = 1$ for all $\tau \in \SYT_{\infty}(\lambda).$

\begin{lem}\label{lem bounds on inv}
    Let $\lambda \in \Par.$ For all $k \geq 0$ there exist only finitely many $\tau \in \SYT_{\infty}(\lambda)$ such that $\mathrm{inv}(\tau) \leq k.$
\end{lem}
\begin{proof}
    This follows from a nearly identical argument to the proof of Corollary 6.6 in \cite{BWMurnaghan}.
\end{proof}

Recall that since $\bF$ is non-archimedean, a series over a countable set $J$ of the form $\sum_{j \in J} c_{j}$ converges if and only if for all $\epsilon> 0$ there exists only finitely many $j \in J$ such that $|c_{j}| > \epsilon.$ This is equivalent to requiring that $\lim_{j \in J} |c_j| = 0$ for any ordering of $J.$

\begin{prop}\label{prop traces converge}
   Let $\lambda \in \Par.$ If $A\in \cB(\hat{S}_{\lambda})$, then the series $\sum_{\tau \in \SYT_{\infty}(\lambda)} (A(\tilde{e}_{\tau}),\tilde{e}_{\tau})_{\lambda}$ converges in $\bF.$ Furthermore, $|\sum_{\tau \in \SYT_{\infty}(\lambda)} (A(\tilde{e}_{\tau}),\tilde{e}_{\tau})_{\lambda} | \leq ||A||.$ 
\end{prop}
\begin{proof}
    Let $(a_{\tau,\gamma})_{\tau, \gamma \in \SYT_{\infty}(\lambda)}$ be the unique scalars such that for all $\tau \in \SYT_{\infty}(\lambda)$, $A(\tilde{e}_{\tau}) = \sum_{\gamma\in \SYT_{\infty}(\lambda)} a_{\tau,\gamma} \tilde{e}_{\gamma}.$ Then for $\tau \in \SYT_{\infty}(\lambda),$ 
    \begin{align*}
        |( A(\tilde{e}_{\tau}),\tilde{e}_{\tau})_{\lambda}| &= |\alpha_{\tau,\tau} a(\tau)|\\
        &\leq |a(\tau)|\sup_{\gamma\in \SYT_{\infty}(\lambda)} |\alpha_{\tau,\gamma}| \\
        &= |a(\tau)| \cdot ||A(\tilde{e}_{\tau})||\\
        &\leq |a(\tau)| \cdot ||A||\\
        &= t^{\mathrm{inv}(\tau)}\cdot ||A||.\\
    \end{align*}
    From Lemma \ref{lem bounds on inv}, for all $k \geq 0$ there exist only finitely many $\tau \in \SYT_{\infty}(\lambda)$ such that $\mathrm{inv}(\tau) \leq k$. Therefore, for all $c > 0$ there exists only finitely many $\tau \in \SYT_{\infty}(\lambda)$ such that $t^{\mathrm{inv}(\tau)}\cdot ||A|| \geq c.$ But since $|( A(\tilde{e}_{\tau}),\tilde{e}_{\tau})_{\lambda}| \leq t^{\mathrm{inv}(\tau)}\cdot ||A||$ for all $\tau \in \SYT_{\infty}(\lambda),$ there exist only finitely many $\tau \in \SYT_{\infty}(\lambda)$ such that $|( A(\tilde{e}_{\tau}),\tilde{e}_{\tau})_{\lambda}| > c.$ Therefore, the series $\sum_{\tau \in \SYT_{\infty}(\lambda)} (A(\tilde{e}_{\tau}),\tilde{e}_{\tau})_{\lambda}$ converges in $\bF.$ Lastly, 
    $$|\sum_{\tau \in \SYT_{\infty}(\lambda)} (A(\tilde{e}_{\tau}),\tilde{e}_{\tau})_{\lambda}| \leq \sup_{\tau \in \SYT_{\infty}(\lambda)} |(A(\tilde{e}_{\tau}),\tilde{e}_{\tau})_{\lambda}| \leq \sup_{\tau \in \SYT_{\infty}(\lambda)} ||A|| \cdot |a(\tau)| \leq ||A||.$$
\end{proof}

\begin{defn}\label{defn regularized trace}
    Let $\lambda \in \Par.$ Define the \textbf{regularized trace} $\Gamma_{\lambda}:\cH_{\infty} \rightarrow \bF$ as the map 
    $$\Gamma_{\lambda}(X):= \sum_{\tau \in \SYT_{\infty}(\lambda)} (X(\tilde{e}_{\tau}),\tilde{e}_{\tau})_{\lambda}.$$
\end{defn}

\begin{remark}\label{remark vector-valued}
    Although the inner product $(-,-)_{\lambda}$ is (essentially) unique, we have made a very specific choice in defining the function $\Gamma_{\lambda}: \cH_{\infty} \rightarrow \bF.$ If we rescale the basis $\tilde{e}_{\tau}$, we would obtain another possible definition for the regularized trace of the representation $\hat{S}_{\lambda}.$ This ambiguity would not be present if we were working over $\mathbb{C}$ since in that case traces of trace-class operators on Hilbert spaces are basis-independent. However, Definition \ref{defn regularized trace} was chosen carefully to align with Theorem 6.12 of the author's prior paper \cite{BWMurnaghan}. In that context, certain sum-product identities were obtained using the combinatorics of vector-valued double affine Hecke algebra representations. The series which appear in that context are nearly identical to the values $\Gamma_{\lambda}(1).$ See Section \ref{section conj} for more details.
\end{remark}

We require the following symmetry of $\cH_{\infty}.$

\begin{lem}
    There exists a unique linear anti-automorphism $\iota: \cH_{\infty} \rightarrow \cH_{\infty}$ such that $\iota(T_i) = T_{i}$ for all $i \geq 1.$
\end{lem}
\begin{proof}
    The set $\{T_{\sigma}| \sigma \in \mathfrak{S}_{\infty}\}$ forms a basis for $\cH_{\infty}.$ Define $\iota(T_{\sigma}):= T_{\sigma^{-1}}$ and extend linearly. It is straightforward to verify that $\iota$ is an anti-automorphism of $\cH_{\infty}$ by checking that $\iota$ preserves the relations defining $\cH_{\infty}.$
\end{proof}

The regularized trace maps are self-adjoint in the following sense.

\begin{prop}
    For all $\lambda \in \Par,$ $\Gamma_{\lambda}\circ \iota = \Gamma_{\lambda}.$
\end{prop}
\begin{proof}
    Let $\sigma \in \mathfrak{S}_{\infty}$ and $\lambda \in \Par.$ From Theorem \ref{thm unique inner product} and Lemma \ref{lem adjoint properties}, we know that $T_{\sigma}^* = T_{\sigma^{-1}}$. Therefore, 
    \begin{align*}
        \Gamma_{\lambda}(T_{\sigma}) &= \sum_{\tau \in \SYT_{\infty}(\lambda)} (T_{\sigma}(\tilde{e}_{\tau}),\tilde{e}_{\tau})_{\lambda}\\
        &= \sum_{\tau \in \SYT_{\infty}(\lambda)} (\tilde{e}_{\tau},T_{\sigma}^*(\tilde{e}_{\tau}))_{\lambda}\\
        &= \sum_{\tau \in \SYT_{\infty}(\lambda)} (\tilde{e}_{\tau},T_{\sigma^{-1}}(\tilde{e}_{\tau}))_{\lambda}\\
        &= \sum_{\tau \in \SYT_{\infty}(\lambda)} (T_{\sigma^{-1}}(\tilde{e}_{\tau}),\tilde{e}_{\tau})_{\lambda}\\
        &= \Gamma_{\lambda}(T_{\sigma^{-1}})\\
        &= \Gamma_{\lambda}(\iota(T_{\sigma})).\\
    \end{align*}
    Therefore, by the linearity of $\Gamma_{\lambda},$ $\Gamma_{\lambda}\circ \iota = \Gamma_{\lambda}.$
\end{proof}

We may find formulas for the values of $\Gamma_{\lambda}$ at certain nice elements of $\cH_{\infty}.$

\begin{prop}\label{prop gamma values at theta}
    Let $a_1,a_2,\dots, a_r \geq 0$ and $\lambda \in \Par$. Then 
    $\Gamma_{\lambda}(\theta_{1}^{a_1}\cdots \theta_{r}^{a_r})$ is explicitly given by the following:
    $$\sum_{\tau \in \SYT_{\infty}}  t^{a_1c_{\tau}(1)+\dots + a_{r}c_{\tau}(r)+\mathrm{inv(\tau)}}\prod_{(\square_1,\square_2) \in \mathrm{Inv}(\tau)} \left( \frac{1-t^{c(\square_2)-c(\square_1)-1}}{1-t^{c(\square_2)-c(\square_1)+1}} \right).$$
\end{prop}
\begin{proof}
    This follows from the simple fact that $\theta_{1}^{a_1}\cdots \theta_{r}^{a_r}(\tilde{e}_{\tau}) = t^{a_1c_1(\tau)+\dots + a_{r}c_{r}(\tau)}\tilde{e}_{\tau}$ and $$(\tilde{e}_{\lambda},\tilde{e}_{\lambda})_{\lambda} =  t^{\mathrm{inv(\tau)}}\prod_{(\square_1,\square_2) \in \mathrm{Inv}(\tau)} \left( \frac{1-t^{c(\square_2)-c(\square_1)-1}}{1-t^{c(\square_2)-c(\square_1)+1}} \right)$$
    for all $\tau \in \SYT_{\infty}$.
\end{proof}

\begin{prop}
    Let $\sigma \in \mathfrak{S}_{\infty}$ be an involution with odd length and $\mathrm{char} \bF \neq 2$. Then for all $\lambda \in \Par$, $\Gamma_{\lambda}(\varphi_{\sigma}) = 0.$
\end{prop}
\begin{proof}
  Using Lemma \ref{lem adjoint properties}, we know that for all $i \geq 1$, since $T_i$ and $\theta_{i}$ are self-adjoint with respect to every $(-,-)_{\lambda}$, $\varphi_i = tT_{i}^{-1}\theta_i - \theta_i tT_{i}^{-1}$ is skew-adjoint with respect to each $(-,-)_{\lambda}.$ Since $\sigma$ is an involution, $
  \varphi_{\sigma}^* = (-1)^{\ell(\sigma)} \varphi_{\sigma^{-1}} = (-1)^{\ell(\sigma)} \varphi_{\sigma} = -\varphi_{\sigma}$ since $\ell(\sigma)$ is odd. Then since $\Gamma_{\lambda} \circ \iota = \Gamma_{\lambda},$ $\Gamma_{\lambda}(\varphi_{\sigma}) = -\Gamma_{\lambda}(\varphi_{\sigma})$ so as $\mathrm{char} \bF \neq 2$, $\Gamma_{\lambda}(\varphi_{\sigma}) = 0.$
\end{proof}

Here we investigate integrality properties of the $\Gamma_{\lambda}$ maps. Let $\mathcal{O}:= \{ x\in \bF |~~ |x| \leq 1 \}.$ Define $\cH_{\infty}^{+}:= \mathcal{O}.\{T_{\sigma} | \sigma \in \mathfrak{S}_{\infty}\}.$

\begin{thm}\label{thm integrality}
    If $X \in \cH_{\infty}^{+},$ then for all $\lambda \in \Par$, $\Gamma_{\lambda}(X) \in \mathcal{O}.$
\end{thm}
\begin{proof}
    It suffices to show that $\Gamma_{\lambda}(T_{\sigma}) \in \mathcal{O}$ for all $\sigma \in \mathfrak{S}_{\infty}$ and $\lambda \in \Par.$ Fix $\lambda \in \Par$ and $\sigma \in \mathfrak{S}_{\infty}$. It is evident from Proposition \ref{prop gamma values at theta} that $\Gamma_{\lambda}(1) \in \mathcal{O}$ so it suffices to show that $\Gamma_{\lambda}(T_{\sigma}-1) \in \mathcal{O}.$ To this end note that $(T_{\sigma}-1)(\tilde{e}_{\tau}) = 0$ for all but finitely many $\tau \in \SYT_{\infty}(\lambda).$ Furthermore, from Remark \ref{remark explicit action}, we see that the structure constants of the action of $\cH_{\infty}$ on $S_{\lambda^{(\infty)}}$ are all in $\mathcal{O}$. More specifically, 
    \begin{itemize}
        \item if $s_{i}(\tau) > \tau,$ then $c_{\tau}(i+1) -c_{\tau}(i) > 1$ so 
        $$T_i(e_{\tau}) = e_{s_{i}(\tau)} + \frac{1-t}{1-t^{c_{\tau}(i+1)-c_{\tau}(i)}} e_{\tau} \in \mathcal{O}.\{e_{\tau}, e_{s_{i}(\tau)}\},$$
        \item if $s_{i}(\tau) < \tau$, then $c_{\tau}(i) -c_{\tau}(i+1) > 1$ so 
        $$T_{i}(e_{\tau}) = t\frac{(1-t^{c_{\tau}(i) - c_{\tau}(i+1) -1})(1-t^{c_{\tau}(i)-c_{\tau}(i+1)+1})}{(1-t^{c_{\tau}(i)-c_{\tau}(i+1)})^2} e_{s_{i}(\tau)} - t^{c_{\tau}(i)-c_{\tau}(i+1)} \frac{1-t}{1-t^{c_{\tau}(i)-c_{\tau}(i+1)}}e_{\tau} \in \mathcal{O}.\{e_{\tau}, e_{s_{i}(\tau)}\},$$
        \item if the labels $i,i+1$ are in the same row in $\tau$, then $T_i(e_{\tau}) = e_{\tau} \in \mathcal{O}.\{e_{\tau}\},$
        \item and if the labels $i,i+1$ are in the same column in $\tau$, then $T_i(e_{\tau}) = -te_{\tau} \in \mathcal{O}\{e_{\tau}\}.$
    \end{itemize}
    Therefore, for all $\tau \in \SYT_{\infty}(\lambda)$, $(T_{\sigma}-1)(e_{\tau}) \in \mathcal{O}.\{e_{\gamma}| \gamma \in \SYT_{\infty}(\lambda)\}$ and similarly, since for all $\gamma \in \SYT_{\infty}(\lambda),$ $\tilde{e}_{\gamma}$ and $e_{\gamma}$ differs by a scalar of norm $1,$ $(T_{\sigma}-1)(\tilde{e}_{\tau}) \in \mathcal{O}.\{\tilde{e}_{\gamma}| \gamma \in \SYT_{\infty}(\lambda)\}.$ Thus, 
    $\Gamma_{\lambda}(T_{\sigma}-1)$ is a finite sum of the form $\sum_{\tau} b_{\tau} (\tilde{e}_{\tau},\tilde{e}_{\tau})_{\lambda}$ for $b_{\tau} \in \mathcal{O}.$ Since $(\tilde{e}_{\tau},\tilde{e}_{\tau})_{\lambda} \in \mathcal{O}$, it follows that $\Gamma_{\lambda}(T_{\sigma}-1) \in \mathcal{O}.$
\end{proof}

\subsection{Rationality Theorem}\label{section conj}

 The values $\Gamma_{\lambda}(X)$ are already interesting in the case $X = 1.$ These values represent a kind of regularized dimension of the representations $\hat{S}_{\lambda}.$ They are given by 
$$\Gamma_{\lambda}(1) = \sum_{\tau \in \SYT_{\infty}(\lambda)} t^{\mathrm{inv}(\tau)} \prod_{(\square_1,\square_2) \in \mathrm{Inv}(\tau)} \left(  \frac{1-t^{c(\square_2)-c(\square_1)-1}}{1-t^{c(\square_2)-c(\square_1)+1}} \right).$$ Note that trivially $\Gamma_{\emptyset}(1) = 1.$ The simplest nontrivial example $\lambda = (1)$ gives 
$$\Gamma_{(1)}(1) = 1 + t\left( \frac{1-t}{1-t^3}\right) + t^2 \left( \frac{1-t}{1-t^3} \right)\left( \frac{1-t^2}{1-t^4} \right)+ t^3 \left( \frac{1-t}{1-t^3} \right)\left( \frac{1-t^2}{1-t^4} \right)\left( \frac{1-t^3}{1-t^5} \right) + \dots$$ which remarkably telescopes to give $\Gamma_{(1)}(1) = 1+t.$ This is very similar to Theorem 6.12 of \cite{BWMurnaghan} but that result, unfortunately, strictly does not apply in this situation. In some sense, the values $\Gamma_{\lambda}(1)$ represent a degenerate case of that result. Nevertheless, the author's previous result suggests that the values $\Gamma_{\lambda}(1)$ should all be rational expressions in the parameter $t.$ That is to say, $\Gamma_{\lambda}(1) \in k(t)$ where $k$ is the smallest field in $\mathbb{F}$ containing $1.$ This is indeed the case. In order to prove this, we require the following definitions.

\begin{defn}
    Let $\lambda \in \Par$ be a partition. A box $\square \in \lambda$ is \textbf{removable} if $\lambda \setminus \square$ is the diagram of a partition. For $\tau \in \SYT_{\infty}(\lambda)$ we define the \textbf{rank} of $\tau$, $\mathrm{rk}(\tau)$, to be the minimum $m \geq n_{\lambda}$ such that the infinite row $\lambda^{(\infty)}/\lambda^{(m)}$ is labeled consecutively by $\tau.$ As a shorthand, we write 
    $\eta(\tau):= t^{\mathrm{inv}(\tau)} \prod_{(\square_1,\square_2) \in \mathrm{Inv}(\tau)} \left(  \frac{1-t^{c(\square_2)-c(\square_1)-1}}{1-t^{c(\square_2)-c(\square_1)+1}} \right).$
\end{defn}

\begin{thm}\label{rationality thm}
    For all $\lambda \in \Par$, $\Gamma_{\lambda}(1) \in k(t).$ 
\end{thm}
\begin{proof}
    Our strategy is to prove a recursive formula for $\Gamma_{\lambda}(1)$ which inductively implies that $\Gamma_{\lambda}(1) \in k(t)$. Let $\lambda \in \Par \setminus \{\emptyset\}.$ Recall that the diagram $\lambda^{(\infty)}$ contains a copy of the diagram $\lambda$ namely, $\lambda \equiv \lambda^{(\infty)}/\emptyset^{(\infty)}$. It is easy to see that for all $\tau \in \SYT_{\infty}(\lambda)$, the unique box in $\lambda^{(\infty)}$ labeled $\mathrm{rk}(\tau)$ by $\tau$ is a removable box in $\lambda$. For each removable box $\square_{0}$ of $\lambda$ define $S_{\square_{0}}$ as the set of $\tau \in \SYT_{\infty}(\lambda)$ such that $\tau(\square_0) = \mathrm{rk}(\tau).$ Necessarily,
    $$\SYT_{\infty}(\lambda) = \bigsqcup_{\square_0~ \text{removable}} S_{\square_0}$$ so 

    $$\Gamma_{\lambda}(1) = \sum_{\tau \in \SYT_{\infty}(\lambda)} \eta(\tau) = \sum_{\square_0~ \text{removable}} \sum_{\tau\in S_{\square_0}} \eta(\tau).$$ 

    Let $\square_0$ be a removable box of $\lambda$ considered in the diagram $\lambda^{(\infty)}$. Given a pair $(m,\tau')$ with $m\geq n_{\lambda}$ and $\tau' \in \SYT_{\infty}(\lambda \setminus \square_0)$ we construct a corresponding $\tau\in \SYT_{\infty}(\lambda)$ by defining $\tau(\square_0) := m$ and filling the boxes of $\lambda^{(\infty)}\setminus \{\square_0\} = (\lambda\setminus \{\square_0\})^{(\infty)}$ using the remaining labels $\{1,2,\dots\} \setminus \{m\}$ following the ordering determined by $\tau'.$ This is clearly a reversible construction so $S_{\square_0}$ is in bijection with $\{n_{\lambda},n_{\lambda}+1,\dots\} \times \SYT_{\infty}(\lambda \setminus \{\square_0\})$ via $(m,\tau') \mapsto \tau.$ It is straightforward casework to determine the following:

    $$\mathrm{Inv}(\tau) = \mathrm{Inv}(\tau') \sqcup \{ (\square_0,\square)| \square \in \lambda^{(\infty)}/\lambda^{(m)}\} \sqcup \{(\square_0,\square)|\square ~ \text{strictly to the right of} ~ \square_0 ~ \text{in} ~ \lambda^{(n_{\lambda})}\}.$$
    From this we directly compute,
    \begin{align*}
        \eta(\tau) &= \eta(\tau') t^{m-n_{\lambda}} \left(\frac{1-t^{\lambda_1-c(\square_0)-1}}{1-t^{\lambda_1-c(\square_0)+1}}\right)\cdots \left(\frac{1-t^{\lambda_1-c(\square_0)-1+m-n_{\lambda}-1}}{1-t^{\lambda_1-c(\square_0)+1+m-n_{\lambda}-1}}\right) \\
        & \times \prod_{\square_0 \rightarrow \square} t^{d(\square)+1}\left(\frac{1-t^{c(\square)-c(\square_0)-1}}{1-t^{c(\square)-c(\square_0)+1}}\right) \cdots \left(\frac{1-t^{c(\square)-c(\square_0)-1+d(\square)}}{1-t^{c(\square)-c(\square_0)+1+d(\square)}}\right)\\
    \end{align*}
    where $\square_0 \rightarrow \square$ ranges over boxes of $\lambda$ which are to the right of $\square_0$ and at the bottom of their column and $d(\square)$ is the number of boxes of $\lambda$ weakly above $\square$ in the same column (i.e., including $\square$ itself).

    Therefore, for all removable $\square_0,$
    \begin{align*}
        &\sum_{\tau\in S_{\square_0}} \eta(\tau) \\
        &=\sum_{m \geq n_{\lambda}} \sum_{\tau' \in \SYT_{\infty}(\lambda\setminus \{\square_0\})}\eta(\tau') t^{m-n_{\lambda}} \left(\frac{1-t^{\lambda_1-c(\square_0)-1}}{1-t^{\lambda_1-c(\square_0)+1}}\right)\cdots \left(\frac{1-t^{\lambda_1-c(\square_0)-1+m-n_{\lambda}}}{1-t^{\lambda_1-c(\square_0)+1+m-n_{\lambda}}}\right) \\
        & \times \prod_{\square_0 \rightarrow \square} t^{d(\square)+1}\left(\frac{1-t^{c(\square)-c(\square_0)-1}}{1-t^{c(\square)-c(\square_0)+1}}\right) \cdots \left(\frac{1-t^{c(\square)-c(\square_0)-1+d(\square)}}{1-t^{c(\square)-c(\square_0)+1+d(\square)}}\right)\\
        &= \Gamma_{\lambda \setminus \{\square_0\}}(1)\prod_{\square_0 \rightarrow \square} t^{d(\square)+1}\left(\frac{1-t^{c(\square)-c(\square_0)-1}}{1-t^{c(\square)-c(\square_0)+1}}\right) \cdots \left(\frac{1-t^{c(\square)-c(\square_0)-1+d(\square)}}{1-t^{c(\square)-c(\square_0)+1+d(\square)}}\right) \\
        & \times \sum_{m \geq n_{\lambda}} t^{m-n_{\lambda}} \left(\frac{1-t^{\lambda_1-c(\square_0)-1}}{1-t^{\lambda_1-c(\square_0)+1}}\right)\cdots \left(\frac{1-t^{\lambda_1-c(\square_0)-1+m-n_{\lambda}}}{1-t^{\lambda_1-c(\square_0)+1+m-n_{\lambda}}}\right).\\
    \end{align*}

Using a standard inductive argument on partial sums one may verify that
$$ \sum_{m \geq n_{\lambda}} t^{m-n_{\lambda}} \left(\frac{1-t^{\lambda_1-c(\square_0)-1}}{1-t^{\lambda_1-c(\square_0)+1}}\right)\cdots \left(\frac{1-t^{\lambda_1-c(\square_0)-1+m-n_{\lambda}}}{1-t^{\lambda_1-c(\square_0)+1+m-n_{\lambda}}}\right) = \frac{1-t^{\lambda_1-c(\square_0)}}{1-t}.$$ Note that here we are computing content values $c(\square)$ in $\lambda^{(n_{\lambda})}.$ If we instead compute the content of $\square \in \lambda$ in the diagram $\lambda$, we need to adjust accordingly. 

In the following formula, all relevant statistics are computed in the the diagram $\lambda$. Putting everything together we find:

\begin{align*}
    &\Gamma_{\lambda}(1) \\
    &= \sum_{\square_0~ \text{removable}} \Gamma_{\lambda \setminus \{\square_0\}}(1) \left( \frac{1-t^{\lambda_1-c(\square_0)+1}}{1-t} \right)\prod_{\square_0 \rightarrow \square} t^{d(\square)+1}\left(\frac{1-t^{c(\square)-c(\square_0)-1}}{1-t^{c(\square)-c(\square_0)+1}}\right) \cdots \left(\frac{1-t^{c(\square)-c(\square_0)-1+d(\square)}}{1-t^{c(\square)-c(\square_0)+1+d(\square)}}\right).\\
\end{align*}

Note that $\Gamma_{\emptyset}(1) = 1 \in k(t).$ Thus by induction on the size of $|\lambda|$ we see that $\Gamma_{\lambda}(1) \in k(t).$

\end{proof}

This results suggests that our choice of normalization (Definition \ref{normalization defn}) was reasonable. Most arbitrary choices for this normalization would not lead to rational values of $\Gamma_{\lambda}(1).$ 

\begin{example}
    
Here we use the above recursive formula to compute the values of $\Gamma_{\lambda}(1)$ for all $|\lambda|\leq 4:$

\begin{itemize}
    \item $\Gamma_{\emptyset}(1) = 1$
    \item $\Gamma_{(1)}(1) = 1+t$
    \item $\Gamma_{(2)}(1)= 1+2t+t^2$
    \item $\Gamma_{(1,1)}(1)= 1+2t+2t^2+t^3$
    \item $\Gamma_{(3)}(1) = 1 + 3 t + 3 t^2 + t^3$
     \item $\Gamma_{(2,1)}(1) = \frac{1+4t+9t^2+13t^3+12t^4+8t^5+4t^6+t^7}{1+t+t^2}$
    \item $\Gamma_{(1,1,1)}(1)= 1 + 3 t + 5 t^2 + 6 t^3 + 5 t^4 + 3 t^5 + t^6$
    \item $\Gamma_{(4)}(1) = 1 + 4 t + 6 t^2 + 4 t^3 + t^4$
    \item $\Gamma_{(3,1)}(1)= \frac{1 + 5 t + 15 t^2 + 33 t^3 + 54 t^4 + 68 t^5 + 68 t^6 + 56 t^7 + 39 t^8 + 21 t^9 + 7 t^{10} + t^{11}}{1+t+2t^2+t^3+t^4}$
    \item $\Gamma_{(2,2)}(1)= 1+4t+9t^2+13t^3+12t^4+8t^5+4t^6+t^7$
    \item $\Gamma_{(2,1,1)}(1) = \frac{1 + 4 t + 10 t^2 + 19 t^3 + 30 t^4 + 41 t^5 + 48 t^6 + 50 t^7 + 47 t^8 + 38 x^9 + 25 t^{10} + 13 t^{11} + 5 t^{12} + t^{13}}{1+t+2t^2+t^3+t^4}$
    \item $\Gamma_{(1,1,1,1)}(1)= 1 + 4 t + 9 t^2 + 15 t^3 + 20 t^4 + 22 t^5 + 20 t^6 + 15 t^7 + 9 t^8 + 4 t^9 + t^{10}.$
\end{itemize} 
\end{example}

As a corollary, we find that the values of $\Gamma_{\lambda}(T_{\sigma})$ are also always rational for all $\sigma \in \mathfrak{S}_{\infty}.$

\begin{cor}
     For all $\lambda \in \Par$ and $\sigma \in \mathfrak{S}_{\infty},$ $\Gamma_{\lambda}(T_{\sigma}) \in k(t).$ 
\end{cor}
\begin{proof}
    Note that $(T_{\sigma} -1)\tilde{e}_{\tau} = 0$ for all but finitely many $\tau \in \SYT_{\infty}(\lambda).$ Recall that the structure coefficients for the action of $\cH_{\infty}$ on $S_{\lambda^{(\infty)}}$ in the $e_{\tau}$-basis are rational functions of $t.$ Therefore, $\Gamma_{\lambda}(T_{\sigma}-1)$ is a finite sum of terms which are guaranteed to be rational functions of $t.$ Hence, $\Gamma_{\lambda}(T_{\sigma}-1) \in k(t).$ From Theorem \ref{rationality thm} $\Gamma_{\lambda}(1) \in k(t)$ so 
    $\Gamma_{\lambda}(T_{\sigma}) = \Gamma_{\lambda}(T_{\sigma}-1) + \Gamma_{\lambda}(1) \in k(t).$
\end{proof}

\printbibliography

\end{document}